\documentclass{article}
\usepackage{amsmath}
\usepackage{amsfonts}
\usepackage{amsthm}
\usepackage{geometry}
\title{Matrix variate and tensor variate Laplace distributions}
\author{Yurii Yurchenko\thanks{Odessa Polytechnic State University, Institute of Computer Systems, Department of Applied Mathematics and Information Technology, Shevchenko av. 1, Odessa 65044, Ukraine}}
\date{2021, May}
\begin{document}
\maketitle
\begin{abstract}
In this article, we define a matrix variate asymmetric Laplace distribution. We prove some properties of the matrix variate asymmetric Laplace distribution. We prove the relationship between the matrix variate asymmetric Laplace distribution and other distributions. We define a matrix variate generalized asymmetric Laplace distribution. We prove some properties of the matrix variate generalized asymmetric Laplace distribution. We prove the relationship between the matrix variate generalized asymmetric Laplace distribution and other distributions. We define a tensor variate asymmetric Laplace distribution. We prove some properties of the tensor variate asymmetric Laplace distribution. We prove the relationship between the tensor variate asymmetric Laplace distribution and other distributions. We define a tensor variate generalized asymmetric Laplace distribution. We prove some properties of the tensor variate generalized asymmetric Laplace distribution. We prove the relationship between the tensor variate generalized asymmetric Laplace distribution and other distributions.
\end{abstract}
\newtheorem{prop}{Proposition}[subsection]
\newtheorem{cor}{Corollary}[subsection]
\newtheorem{defi}{Definition}[subsection]
\newtheorem{property}{Property}[subsection]
\newtheorem{lemma}{Lemma}[subsection]
\newtheorem{theorem}{Theorem}[subsection]
\section{Matrix variate Laplace distributions}
\subsection{Introduction}
The matrix variate asymmetric Laplace distribution is a continuous probability distribution that is a generalization of the multivariate asymmetric Laplace distribution to matrix-valued random variables. The matrix variate asymmetric Laplace distribution of a random matrix $\mathbf{X}\in\mathbb{R}^{k\times n}$ can be written in the following notation:
\begin{equation*}
\mathcal{MAL}_{k\times n}\left(\mathbf{M},\mathbf{\Sigma},\mathbf{\Psi}\right)
\end{equation*}
or
\begin{equation*}
\mathcal{MAL}\left(\mathbf{M},\mathbf{\Sigma},\mathbf{\Psi}\right),
\end{equation*}
where $\mathbf{M}\in \mathbb{R}^{k\times n}$ is a location matrix, $\mathbf{\Sigma}\in \mathbb{R}^{k\times k}$ is a positive-definite scale matrix and $\mathbf{\Psi}\in \mathbb{R}^{n\times n}$ is a positive-definite scale matrix.

The matrix variate generalized asymmetric Laplace distribution is a continuous probability distribution that is a generalization of the multivariate generalized asymmetric Laplace distribution to matrix-valued random variables. The matrix variate generalized asymmetric Laplace distribution of a random matrix $\mathbf{X}\in\mathbb{R}^{k\times n}$ can be written in the following notation:
\begin{equation*}
\mathcal{MGAL}_{k\times n}\left(\mathbf{M},\mathbf{\Sigma},\mathbf{\Psi},\lambda\right)
\end{equation*}
or
\begin{equation*}
\mathcal{MGAL}\left(\mathbf{M},\mathbf{\Sigma},\mathbf{\Psi},\lambda\right),
\end{equation*}
where $\mathbf{M}\in \mathbb{R}^{k\times n}$ is a location matrix, $\mathbf{\Sigma}\in \mathbb{R}^{k\times k}$ is a positive-definite scale matrix, $\mathbf{\Psi}\in \mathbb{R}^{n\times n}$ is a positive-definite scale matrix and $\lambda>0$. 
\subsection{Matrix variate asymmetric Laplace distribution}
\begin{defi}
\label{pmfmataslap}
$\mathbf{X}\sim\mathcal{MAL}\left(\mathbf{M},\mathbf{\Sigma},\mathbf{\Psi}\right)$ if and only if the probability density function of the random matrix $\mathbf{X}\in\mathbb{R}^{k\times n}$ is given by
\begin{equation*}
\begin{split}
f_{\mathbf{X}}\left(\mathbf{x}\right)&=\frac{2\operatorname{exp}\operatorname{tr}\left(\mathbf{\Psi}^{-1}\mathbf{x}^{\mathrm{T}}\mathbf{\Sigma}^{-1}\mathbf{M}\right)}{\left(2\pi\right)^{kn/2}\left(\det\mathbf{\Psi}\right)^{k/2}\left(\det\mathbf{\Sigma}\right)^{n/2}}\left(\frac{\operatorname{tr}\left(\mathbf{\Psi}^{-1}\mathbf{x}^{\mathrm{T}}\mathbf{\Sigma}^{-1}\mathbf{x}\right)}{2+\operatorname{tr}\left(\mathbf{\Psi}^{-1}\mathbf{M}^{\mathrm{T}}\mathbf{\Sigma}^{-1}\mathbf{M}\right)}\right)^{1/2 -kn/4}\\&\times K_{1 -kn/2}\left(\sqrt{\left(2+\operatorname{tr}\left(\mathbf{\Psi}^{-1}\mathbf{M}^{\mathrm{T}}\mathbf{\Sigma}^{-1}\mathbf{M}\right)\right)\operatorname{tr}\left(\mathbf{\Psi}^{-1}\mathbf{x}^{\mathrm{T}}\mathbf{\Sigma}^{-1}\mathbf{x}\right)}\right),
\end{split}
\end{equation*}
where  $K_{1 -kn/2}(\cdot)$ is the modified Bessel function of the third kind.
\end{defi}
\begin{theorem}
\label{vecaslap}
$\mathbf{X}\sim\mathcal{MAL}_{k\times n}\left(\mathbf{M},\mathbf{\Sigma},\mathbf{\Psi}\right)\Leftrightarrow\operatorname{vec}\left(\mathbf{X}\right)\sim\mathcal{AL}_{kn}\left(\operatorname{vec}\mathbf{M},\mathbf{\Psi}\otimes\mathbf{\Sigma}\right).$
\end{theorem}
\begin{proof}
By Definition \ref{pmfmataslap}, $\mathbf{X}\sim\mathcal{MAL}_{k\times n}\left(\mathbf{M},\mathbf{\Sigma},\mathbf{\Psi}\right)$ if and only if
\begin{equation*}
\begin{split}
f_{\mathbf{X}}\left(\mathbf{x}\right)&=\frac{2\operatorname{exp}\operatorname{tr}\left(\mathbf{\Psi}^{-1}\mathbf{x}^{\mathrm{T}}\mathbf{\Sigma}^{-1}\mathbf{M}\right)}{\left(2\pi\right)^{kn/2}\left(\det\mathbf{\Psi}\right)^{k/2}\left(\det\mathbf{\Sigma}\right)^{n/2}}\left(\frac{\operatorname{tr}\left(\mathbf{\Psi}^{-1}\mathbf{x}^{\mathrm{T}}\mathbf{\Sigma}^{-1}\mathbf{x}\right)}{2+\operatorname{tr}\left(\mathbf{\Psi}^{-1}\mathbf{M}^{\mathrm{T}}\mathbf{\Sigma}^{-1}\mathbf{M}\right)}\right)^{1/2 -kn/4}\\&\times K_{1 -kn/2}\left(\sqrt{\left(2+\operatorname{tr}\left(\mathbf{\Psi}^{-1}\mathbf{M}^{\mathrm{T}}\mathbf{\Sigma}^{-1}\mathbf{M}\right)\right)\operatorname{tr}\left(\mathbf{\Psi}^{-1}\mathbf{x}^{\mathrm{T}}\mathbf{\Sigma}^{-1}\mathbf{x}\right)}\right),
\end{split}
\end{equation*}
where  $K_{1 -kn/2}(\cdot)$ is the modified Bessel function of the third kind.

By the definition of the multivariate asymmetric Laplace distribution \cite{kotzas}, $\operatorname{vec}\left(\mathbf{X}\right)\sim\mathcal{AL}_{kn}\left(\operatorname{vec}\mathbf{M},\mathbf{\Psi}\otimes\mathbf{\Sigma}\right)$ if and only if
\begin{equation*}
\begin{split}
 f_{\operatorname{vec}\mathbf{X}}\left(\operatorname{vec}\mathbf{x}\right)&=\frac{2\operatorname{exp}\left(\left(\operatorname{vec}\mathbf{x}\right)^{\mathrm{T}}\left(\mathbf{\Psi}\otimes\mathbf{\Sigma}\right)^{-1}\left(\operatorname{vec}\mathbf{M}\right)\right)}{\left(2\pi\right)^{kn/2}\left(\det\left(\mathbf{\Psi}\otimes\mathbf{\Sigma}\right)\right)^{1/2}}\left(\frac{\left(\operatorname{vec}\mathbf{x}\right)^{\mathrm{T}}\left(\mathbf{\Psi}\otimes\mathbf{\Sigma}\right)^{-1}\left(\operatorname{vec}\mathbf{x}\right)}{2+\left(\operatorname{vec}\mathbf{M}\right)^{\mathrm{T}}\left(\mathbf{\Psi}\otimes\mathbf{\Sigma}\right)^{-1}\left(\operatorname{vec}\mathbf{M}\right)}\right)^{1/2 -kn/4}\\& \times K_{1 -kn/2}\left(\sqrt{\left(2+\left(\operatorname{vec}\mathbf{M}\right)^{\mathrm{T}}\left(\mathbf{\Psi}\otimes\mathbf{\Sigma}\right)^{-1}\left(\operatorname{vec}\mathbf{M}\right)\right)\left(\left(\operatorname{vec}\mathbf{x}\right)^{\mathrm{T}}\left(\mathbf{\Psi}\otimes\mathbf{\Sigma}\right)^{-1}\left(\operatorname{vec}\mathbf{x}\right)\right)}\right),
\end{split}
\end{equation*}
where  $K_{1 -kn/2}(\cdot)$ is the modified Bessel function of the third kind.

By the properties of the trace, Kronecker product and vectorization \cite[p.~2-12]{gupta},
\begin{equation*}
\begin{split}
&\operatorname{tr}\left(\mathbf{\Psi}^{-1}\mathbf{A}^{\mathrm{T}}\mathbf{\Sigma}^{-1}\mathbf{B}\right)\\&
=\left(\operatorname{vec}\mathbf{A}\right)^{\mathrm{T}}\operatorname{vec}\left(\mathbf{\Sigma}^{-1}\mathbf{B}\mathbf{\Psi}^{-1}\right)\\&
=\left(\operatorname{vec}\mathbf{A}\right)^{\mathrm{T}}\left(\left(\mathbf{\Psi}^{-1}\right)^{\mathrm{T}}\otimes\mathbf{\Sigma}^{-1}\right)\left(\operatorname{vec}\mathbf{B}\right)\\&
=\left(\operatorname{vec}\mathbf{A}\right)^{\mathrm{T}}\left(\mathbf{\Psi}^{\mathrm{T}}\otimes\mathbf{\Sigma}\right)^{-1}\left(\operatorname{vec}\mathbf{B}\right)\\&
=\left(\operatorname{vec}\mathbf{A}\right)^{\mathrm{T}}\left(\mathbf{\Psi}\otimes\mathbf{\Sigma}\right)^{-1}\left(\operatorname{vec}\mathbf{B}\right).
\end{split}
\end{equation*}
Therefore,
\begin{equation*}
\begin{split}
\operatorname{tr}\left(\mathbf{\Psi}^{-1}\mathbf{x}^{\mathrm{T}}\mathbf{\Sigma}^{-1}\mathbf{x}\right)=\left(\operatorname{vec}\mathbf{x}\right)^{\mathrm{T}}\left(\mathbf{\Psi}\otimes\mathbf{\Sigma}\right)^{-1}\left(\operatorname{vec}\mathbf{x}\right),
\end{split}
\end{equation*}
\begin{equation*}
\begin{split}
\operatorname{tr}\left(\mathbf{\Psi}^{-1}\mathbf{M}^{\mathrm{T}}\mathbf{\Sigma}^{-1}\mathbf{M}\right)=\left(\operatorname{vec}\mathbf{M}\right)^{\mathrm{T}}\left(\mathbf{\Psi}\otimes\mathbf{\Sigma}\right)^{-1}\left(\operatorname{vec}\mathbf{M}\right),
\end{split}
\end{equation*}
and
\begin{equation*}
\begin{split}
\operatorname{tr}\left(\mathbf{\Psi}^{-1}\mathbf{x}^{\mathrm{T}}\mathbf{\Sigma}^{-1}\mathbf{M}\right)=\left(\operatorname{vec}\mathbf{x}\right)^{\mathrm{T}}\left(\mathbf{\Psi}\otimes\mathbf{\Sigma}\right)^{-1}\left(\operatorname{vec}\mathbf{M}\right).
\end{split}
\end{equation*}
By the property of the determinant \cite[p.~2-12]{gupta},
\begin{equation*}
\begin{split}
\left(\det\mathbf{\Psi}\right)^{k/2}\left(\det\mathbf{\Sigma}\right)^{n/2}
=\left(\det\left(\mathbf{\Psi}\otimes\mathbf{\Sigma}\right)\right)^{1/2}.
\end{split}
\end{equation*}
Thus, $\mathbf{X}\sim\mathcal{MAL}_{k\times n}\left(\mathbf{M},\mathbf{\Sigma},\mathbf{\Psi}\right)\Leftrightarrow\operatorname{vec}\left(\mathbf{X}\right)\sim\mathcal{AL}_{kn}\left(\operatorname{vec}\mathbf{M},\mathbf{\Psi}\otimes\mathbf{\Sigma}\right).$
\end{proof}
\begin{theorem}
\label{charmataslap}
Let $\mathbf{T}\in \mathbb{R}^{k\times n}$. $\mathbf{X}\sim\mathcal{MAL}_{k\times n}\left(\mathbf{M},\mathbf{\Sigma},\mathbf{\Psi}\right)$ if and only if the characteristic function of the random matrix $\mathbf{X}\in\mathbb{R}^{k\times n}$  is given by
\begin{equation*}
\varphi_{\mathbf{X}}\left(\mathbf{T}\right)=\frac{1}{1+\frac{1}{2}\operatorname{tr}\left(\mathbf{\Psi}\mathbf{T}^{\mathrm{T}}\mathbf{\Sigma}\mathbf{T}\right)-\mathrm{i}\operatorname{tr}\left(\mathbf{M}^{\mathrm{T}}\mathbf{T}\right)}.
\end{equation*}
\end{theorem}
\begin{proof}
By Theorem \ref{vecaslap},
\begin{equation*}
\mathbf{X}\sim\mathcal{MAL}_{k\times n}\left(\mathbf{M},\mathbf{\Sigma},\mathbf{\Psi}\right)\Leftrightarrow\operatorname{vec}\left(\mathbf{X}\right)\sim\mathcal{AL}_{kn}\left(\operatorname{vec}\mathbf{M},\mathbf{\Psi}\otimes\mathbf{\Sigma}\right).
\end{equation*}
By the definition of the multivariate asymmetric Laplace distribution \cite{kotzas},
\begin{equation*}
\varphi_{\operatorname{vec}\mathbf{X}}\left(\operatorname{vec}\mathbf{T}\right)=\frac{1}{1+\frac{1}{2}\left(\operatorname{vec}\mathbf{T}\right)^{\mathrm{T}}\left(\mathbf{\Psi}\otimes\mathbf{\Sigma}\right)\left(\operatorname{vec}\mathbf{T}\right)-\mathrm{i}\left(\operatorname{vec}\mathbf{M}\right)^{\mathrm{T}}\left(\operatorname{vec}\mathbf{T}\right)}.
\end{equation*}
Thus,
\begin{equation*}
\varphi_{\operatorname{vec}\mathbf{X}}\left(\operatorname{vec}\mathbf{T}\right)=\varphi_{\mathbf{X}}\left(\mathbf{T}\right)=\frac{1}{1+\frac{1}{2}\operatorname{tr}\left(\mathbf{\Psi}\mathbf{T}^{\mathrm{T}}\mathbf{\Sigma}\mathbf{T}\right)-\mathrm{i}\operatorname{tr}\left(\mathbf{M}^{\mathrm{T}}\mathbf{T}\right)}.
\end{equation*}
\end{proof}
\begin{theorem}
Let $\mathbf{Y}\sim\mathcal{MAL}_{k\times n}\left(\mathbf{M},\mathbf{\Sigma},\mathbf{\Psi}\right)$, $\mathbf{X}\sim\mathcal{MN}_{k\times n}\left(\mathbf{0},\mathbf{\Sigma},\mathbf{\Psi}\right)$ and $W\sim\mathrm{Exp}\left(1\right)$, independent of $\mathbf{X}$, then 
\begin{equation*}
\mathbf{Y}=\mathbf{M}W+W^{1/2}\mathbf{X}.
\end{equation*}
\end{theorem}
\begin{proof}
Let $\mathbf{Y}=\mathbf{M}W+W^{1/2}\mathbf{X}$.

Therefore,
\begin{equation*}
\operatorname{vec}\mathbf{Y}=\operatorname{vec}\left(\mathbf{M}W+W^{1/2}\mathbf{X}\right)=\left(\operatorname{vec}\mathbf{M}\right)W+W^{1/2}\operatorname{vec}\mathbf{X}.
\end{equation*}
By the definition of the matrix normal distribution \cite{gupta},
\begin{equation*}
\mathbf{X}\sim\mathcal{MN}_{k\times n}\left(\mathbf{M},\mathbf{\Sigma},\mathbf{\Psi}\right)\Leftrightarrow\operatorname{vec}\left(\mathbf{X}\right)\sim\mathcal{N}_{kn}\left(\operatorname{vec}\mathbf{M},\mathbf{\Psi}\otimes\mathbf{\Sigma}\right).
\end{equation*}
By the theorem \cite[Theorem 6.3.1]{kotzas},
\begin{equation*}
\operatorname{vec}\left(\mathbf{Y}\right)\sim\mathcal{AL}_{kn}\left(\operatorname{vec}\mathbf{M},\mathbf{\Psi}\otimes\mathbf{\Sigma}\right).
\end{equation*}
By Theorem \ref{vecaslap},
\begin{equation*}
\mathbf{Y}\sim\mathcal{MAL}_{k\times n}\left(\mathbf{M},\mathbf{\Sigma},\mathbf{\Psi}\right)\Leftrightarrow\operatorname{vec}\left(\mathbf{Y}\right)\sim\mathcal{AL}_{kn}\left(\operatorname{vec}\mathbf{M},\mathbf{\Psi}\otimes\mathbf{\Sigma}\right).
\end{equation*}
\end{proof}
\begin{theorem}
\label{expvalueaslap}
If $\mathbf{X}\sim\mathcal{MAL}\left(\mathbf{M},\mathbf{\Sigma},\mathbf{\Psi}\right)$, then the expected value of the random matrix $\mathbf{X}\in\mathbb{R}^{k\times n}$ is given by
\begin{equation*}
 \mathbb{E}\left[\mathbf{X}\right]=\mathbf{M}.
\end{equation*}
\end{theorem}
\begin{proof}
By Theorem \ref{vecaslap},
\begin{equation*}
\mathbf{X}\sim\mathcal{MAL}_{k\times n}\left(\mathbf{M},\mathbf{\Sigma},\mathbf{\Psi}\right)\Leftrightarrow\operatorname{vec}\left(\mathbf{X}\right)\sim\mathcal{AL}_{kn}\left(\operatorname{vec}\mathbf{M},\mathbf{\Psi}\otimes\mathbf{\Sigma}\right).
\end{equation*}
By the property of the multivariate asymmetric Laplace distribution \cite{kotzas},
\begin{equation*}
 \mathbb{E}\left[\operatorname{vec}\mathbf{X}\right]=\operatorname{vec}\mathbf{M}.
\end{equation*}
Thus,
\begin{equation*}
 \mathbb{E}\left[\mathbf{X}\right]=\mathbf{M}.
\end{equation*}
\end{proof}
\begin{theorem}
\label{kxxaslap}
If  $\mathbf{X}\sim\mathcal{MAL}_{k\times n}\left(\mathbf{M},\mathbf{\Sigma},\mathbf{\Psi}\right)$, then the variance-covariance matrix of the random matrix $\mathbf{X}\in\mathbb{R}^{k\times n}$ is given by
\begin{equation*}
\mathbf{K}_{\mathbf{XX}}=\mathbf{\Psi}\otimes\mathbf{\Sigma}+\left(\operatorname{vec}\mathbf{M}\right)\left(\operatorname{vec}\mathbf{M}\right)^{\mathrm{T}}.
\end{equation*}
\end{theorem}
\begin{proof}
By Theorem \ref{vecaslap},
\begin{equation*}
\mathbf{X}\sim\mathcal{MAL}_{k\times n}\left(\mathbf{M},\mathbf{\Sigma},\mathbf{\Psi}\right)\Leftrightarrow\operatorname{vec}\left(\mathbf{X}\right)\sim\mathcal{AL}_{kn}\left(\operatorname{vec}\mathbf{M},\mathbf{\Psi}\otimes\mathbf{\Sigma}\right).
\end{equation*}
By the property of the multivariate asymmetric Laplace distribution \cite{kotzas},
\begin{equation*}
\mathbf{K}_{\operatorname{vec}\mathbf{X},\operatorname{vec}\mathbf{X}}=\mathbf{\Psi}\otimes\mathbf{\Sigma}+\left(\operatorname{vec}\mathbf{M}\right)\left(\operatorname{vec}\mathbf{M}\right)^{\mathrm{T}}.
\end{equation*}
Thus,
\begin{equation*}
\mathbf{K}_{\mathbf{XX}}=\mathbf{\Psi}\otimes\mathbf{\Sigma}+\left(\operatorname{vec}\mathbf{M}\right)\left(\operatorname{vec}\mathbf{M}\right)^{\mathrm{T}}.
\end{equation*}
\end{proof}
\begin{theorem}
Let $\mathbf{D}$ is a rectangular $m\times k$ matrix, $\mathbf{C}$ is a rectangular $n\times p$ matrix, $m\leq k$ and $p\leq n$, then
\begin{equation*}
\mathbf{X}\sim\mathcal{MAL}_{k\times n}\left(\mathbf{M},\mathbf{\Sigma},\mathbf{\Psi}\right)\Leftrightarrow\mathbf{DXC}\sim\mathcal{MAL}_{m\times p}\left(\mathbf{DMC},\mathbf{D\Sigma}\mathbf{D}^{\mathrm{T}},\mathbf{C}^{\mathrm{T}}\mathbf{\Psi C}\right).
\end{equation*}
\end{theorem}
\begin{proof}
The characteristic function of $\mathbf{DXC}$ is
\begin{equation*}
\varphi_{\mathbf{DXC}}\left(\mathbf{T}\right)=\mathbb{E}\left[e^{\mathrm{i}\operatorname{tr}\left(\mathbf{DXC}\mathbf{T}^{\mathrm{T}}\right)}\right]=\mathbb{E}\left[e^{\mathrm{i}\operatorname{tr}\left(\mathbf{X}\mathbf{C}\mathbf{T}^{\mathrm{T}}\mathbf{D}\right)}\right].
\end{equation*}
By Theorem \ref{charmataslap},
\begin{equation*}
\begin{split}
\varphi_{\mathbf{DXC}}\left(\mathbf{T}\right)&=\frac{1}{1+\frac{1}{2}\operatorname{tr}\left(\mathbf{\Psi}\mathbf{C}\mathbf{T}^{\mathrm{T}}\mathbf{D}\mathbf{\Sigma}\mathbf{D}^{\mathrm{T}}\mathbf{T}\mathbf{C}^{\mathrm{T}}\right)-\mathrm{i}\operatorname{tr}\left(\mathbf{C}\mathbf{T}^{\mathrm{T}}\mathbf{D}\mathbf{M}\right)}\\&=\frac{1}{1+\frac{1}{2}\operatorname{tr}\left(\mathbf{C}^{\mathrm{T}}\mathbf{\Psi}\mathbf{C}\mathbf{T}^{\mathrm{T}}\mathbf{D}\mathbf{\Sigma}\mathbf{D}^{\mathrm{T}}\mathbf{T}\right)-\mathrm{i}\operatorname{tr}\left(\mathbf{T}^{\mathrm{T}}\mathbf{D}\mathbf{M}\mathbf{C}\right)}.
\end{split}
\end{equation*}
Thus,
\begin{equation*}
\mathbf{X}\sim\mathcal{MAL}_{k\times n}\left(\mathbf{M},\mathbf{\Sigma},\mathbf{\Psi}\right)\Leftrightarrow\mathbf{DXC}\sim\mathcal{MAL}_{m\times p}\left(\mathbf{DMC},\mathbf{D\Sigma}\mathbf{D}^{\mathrm{T}},\mathbf{C}^{\mathrm{T}}\mathbf{\Psi C}\right).
\end{equation*}
\end{proof}
\begin{cor}
Let $\mathbf{D}$ is a rectangular $m\times k$ matrix and $m\leq k$, then
\begin{equation*}
\mathbf{X}\sim\mathcal{MAL}_{k\times n}\left(\mathbf{M},\mathbf{\Sigma},\mathbf{\Psi}\right)\Leftrightarrow\mathbf{DX}\sim\mathcal{MAL}_{m\times n}\left(\mathbf{DM},\mathbf{D\Sigma}\mathbf{D}^{\mathrm{T}},\mathbf{\Psi}\right).
\end{equation*}
\end{cor}
\begin{cor}
Let $\mathbf{C}$ is a rectangular $n\times p$ matrix and $p\leq n$, then
\begin{equation*}
\mathbf{X}\sim\mathcal{MAL}_{k\times n}\left(\mathbf{M},\mathbf{\Sigma},\mathbf{\Psi}\right)\Leftrightarrow\mathbf{XC}\sim\mathcal{MAL}_{k\times p}\left(\mathbf{MC},\mathbf{\Sigma},\mathbf{C}^{\mathrm{T}}\mathbf{\Psi C}\right).
\end{equation*}
\end{cor}
\subsection{Matrix variate generalized asymmetric Laplace distribution}
\begin{defi}
\label{pmfmataslapg}
$\mathbf{X}\sim\mathcal{MGAL}\left(\mathbf{M},\mathbf{\Sigma},\mathbf{\Psi},\lambda\right)$ if and only if the probability density function of the random matrix $\mathbf{X}\in\mathbb{R}^{k\times n}$ is given by
\begin{equation*}
\begin{split}
f_{\mathbf{X}}\left(\mathbf{x}\right)&=\frac{2\operatorname{exp}\operatorname{tr}\left(\mathbf{\Psi}^{-1}\mathbf{x}^{\mathrm{T}}\mathbf{\Sigma}^{-1}\mathbf{M}\right)}{\left(2\pi\right)^{kn/2}\Gamma\left(\lambda\right)\left(\det\mathbf{\Psi}\right)^{k/2}\left(\det\mathbf{\Sigma}\right)^{n/2}}\left(\frac{\operatorname{tr}\left(\mathbf{\Psi}^{-1}\mathbf{x}^{\mathrm{T}}\mathbf{\Sigma}^{-1}\mathbf{x}\right)}{2+\operatorname{tr}\left(\mathbf{\Psi}^{-1}\mathbf{M}^{\mathrm{T}}\mathbf{\Sigma}^{-1}\mathbf{M}\right)}\right)^{\lambda/2 -kn/4}\\&\times K_{\lambda -kn/2}\left(\sqrt{\left(2+\operatorname{tr}\left(\mathbf{\Psi}^{-1}\mathbf{M}^{\mathrm{T}}\mathbf{\Sigma}^{-1}\mathbf{M}\right)\right)\operatorname{tr}\left(\mathbf{\Psi}^{-1}\mathbf{x}^{\mathrm{T}}\mathbf{\Sigma}^{-1}\mathbf{x}\right)}\right),
\end{split}
\end{equation*}
where  $K_{\lambda -kn/2}(\cdot)$ is the modified Bessel function of the third kind.
\end{defi}
\begin{theorem}
\label{vecaslapg}
$\mathbf{X}\sim\mathcal{MGAL}_{k\times n}\left(\mathbf{M},\mathbf{\Sigma},\mathbf{\Psi},\lambda\right)\Leftrightarrow\operatorname{vec}\left(\mathbf{X}\right)\sim\mathcal{GAL}_{kn}\left(\operatorname{vec}\mathbf{M},\mathbf{\Psi}\otimes\mathbf{\Sigma},\lambda\right).$
\end{theorem}
\begin{proof}
By Definition \ref{pmfmataslapg}, $\mathbf{X}\sim\mathcal{MGAL}_{k\times n}\left(\mathbf{M},\mathbf{\Sigma},\mathbf{\Psi},s\right)$ if and only if
\begin{equation*}
\begin{split}
f_{\mathbf{X}}\left(\mathbf{x}\right)&=\frac{2\operatorname{exp}\operatorname{tr}\left(\mathbf{\Psi}^{-1}\mathbf{x}^{\mathrm{T}}\mathbf{\Sigma}^{-1}\mathbf{M}\right)}{\left(2\pi\right)^{kn/2}\Gamma\left(\lambda\right)\left(\det\mathbf{\Psi}\right)^{k/2}\left(\det\mathbf{\Sigma}\right)^{n/2}}\left(\frac{\operatorname{tr}\left(\mathbf{\Psi}^{-1}\mathbf{x}^{\mathrm{T}}\mathbf{\Sigma}^{-1}\mathbf{x}\right)}{2+\operatorname{tr}\left(\mathbf{\Psi}^{-1}\mathbf{M}^{\mathrm{T}}\mathbf{\Sigma}^{-1}\mathbf{M}\right)}\right)^{\lambda/2 -kn/4}\\&\times K_{\lambda -kn/2}\left(\sqrt{\left(2+\operatorname{tr}\left(\mathbf{\Psi}^{-1}\mathbf{M}^{\mathrm{T}}\mathbf{\Sigma}^{-1}\mathbf{M}\right)\right)\operatorname{tr}\left(\mathbf{\Psi}^{-1}\mathbf{x}^{\mathrm{T}}\mathbf{\Sigma}^{-1}\mathbf{x}\right)}\right),
\end{split}
\end{equation*}
where  $K_{\lambda -kn/2}(\cdot)$ is the modified Bessel function of the third kind.

By the definition of the multivariate asymmetric Laplace distribution \cite{kozub}, $\operatorname{vec}\left(\mathbf{X}\right)\sim\mathcal{GAL}_{kn}\left(\operatorname{vec}\mathbf{M},\mathbf{\Psi}\otimes\mathbf{\Sigma},s\right)$ if and only if
\begin{equation*}
\begin{split}
 f_{\operatorname{vec}\mathbf{X}}\left(\operatorname{vec}\mathbf{x}\right)&=\frac{2\operatorname{exp}\left(\left(\operatorname{vec}\mathbf{x}\right)^{\mathrm{T}}\left(\mathbf{\Psi}\otimes\mathbf{\Sigma}\right)^{-1}\left(\operatorname{vec}\mathbf{M}\right)\right)}{\left(2\pi\right)^{kn/2}\Gamma\left(\lambda\right)\left(\det\left(\mathbf{\Psi}\otimes\mathbf{\Sigma}\right)\right)^{1/2}}\left(\frac{\left(\operatorname{vec}\mathbf{x}\right)^{\mathrm{T}}\left(\mathbf{\Psi}\otimes\mathbf{\Sigma}\right)^{-1}\left(\operatorname{vec}\mathbf{x}\right)}{2+\left(\operatorname{vec}\mathbf{M}\right)^{\mathrm{T}}\left(\mathbf{\Psi}\otimes\mathbf{\Sigma}\right)^{-1}\left(\operatorname{vec}\mathbf{M}\right)}\right)^{\lambda/2 -kn/4}\\& \times K_{\lambda -kn/2}\left(\sqrt{\left(2+\left(\operatorname{vec}\mathbf{M}\right)^{\mathrm{T}}\left(\mathbf{\Psi}\otimes\mathbf{\Sigma}\right)^{-1}\left(\operatorname{vec}\mathbf{M}\right)\right)\left(\left(\operatorname{vec}\mathbf{x}\right)^{\mathrm{T}}\left(\mathbf{\Psi}\otimes\mathbf{\Sigma}\right)^{-1}\left(\operatorname{vec}\mathbf{x}\right)\right)}\right),
\end{split}
\end{equation*}
where  $K_{\lambda -kn/2}(\cdot)$ is the modified Bessel function of the third kind.

By the properties of the trace, Kronecker product and vectorization \cite[p.~2-12]{gupta},
\begin{equation*}
\begin{split}
&\operatorname{tr}\left(\mathbf{\Psi}^{-1}\mathbf{A}^{\mathrm{T}}\mathbf{\Sigma}^{-1}\mathbf{B}\right)\\&
=\left(\operatorname{vec}\mathbf{A}\right)^{\mathrm{T}}\operatorname{vec}\left(\mathbf{\Sigma}^{-1}\mathbf{B}\mathbf{\Psi}^{-1}\right)\\&
=\left(\operatorname{vec}\mathbf{A}\right)^{\mathrm{T}}\left(\left(\mathbf{\Psi}^{-1}\right)^{\mathrm{T}}\otimes\mathbf{\Sigma}^{-1}\right)\left(\operatorname{vec}\mathbf{B}\right)\\&
=\left(\operatorname{vec}\mathbf{A}\right)^{\mathrm{T}}\left(\mathbf{\Psi}^{\mathrm{T}}\otimes\mathbf{\Sigma}\right)^{-1}\left(\operatorname{vec}\mathbf{B}\right)\\&
=\left(\operatorname{vec}\mathbf{A}\right)^{\mathrm{T}}\left(\mathbf{\Psi}\otimes\mathbf{\Sigma}\right)^{-1}\left(\operatorname{vec}\mathbf{B}\right).
\end{split}
\end{equation*}
Therefore,
\begin{equation*}
\begin{split}
\operatorname{tr}\left(\mathbf{\Psi}^{-1}\mathbf{x}^{\mathrm{T}}\mathbf{\Sigma}^{-1}\mathbf{x}\right)=\left(\operatorname{vec}\mathbf{x}\right)^{\mathrm{T}}\left(\mathbf{\Psi}\otimes\mathbf{\Sigma}\right)^{-1}\left(\operatorname{vec}\mathbf{x}\right),
\end{split}
\end{equation*}
\begin{equation*}
\begin{split}
\operatorname{tr}\left(\mathbf{\Psi}^{-1}\mathbf{M}^{\mathrm{T}}\mathbf{\Sigma}^{-1}\mathbf{M}\right)=\left(\operatorname{vec}\mathbf{M}\right)^{\mathrm{T}}\left(\mathbf{\Psi}\otimes\mathbf{\Sigma}\right)^{-1}\left(\operatorname{vec}\mathbf{M}\right),
\end{split}
\end{equation*}
and
\begin{equation*}
\begin{split}
\operatorname{tr}\left(\mathbf{\Psi}^{-1}\mathbf{x}^{\mathrm{T}}\mathbf{\Sigma}^{-1}\mathbf{M}\right)=\left(\operatorname{vec}\mathbf{x}\right)^{\mathrm{T}}\left(\mathbf{\Psi}\otimes\mathbf{\Sigma}\right)^{-1}\left(\operatorname{vec}\mathbf{M}\right).
\end{split}
\end{equation*}
By the property of the determinant \cite[p.~2-12]{gupta},
\begin{equation*}
\begin{split}
\left(\det\mathbf{\Psi}\right)^{k/2}\left(\det\mathbf{\Sigma}\right)^{n/2}
=\left(\det\left(\mathbf{\Psi}\otimes\mathbf{\Sigma}\right)\right)^{1/2}.
\end{split}
\end{equation*}
Thus, $\mathbf{X}\sim\mathcal{MGAL}_{k\times n}\left(\mathbf{M},\mathbf{\Sigma},\mathbf{\Psi},\lambda\right)\Leftrightarrow\operatorname{vec}\left(\mathbf{X}\right)\sim\mathcal{GAL}_{kn}\left(\operatorname{vec}\mathbf{M},\mathbf{\Psi}\otimes\mathbf{\Sigma},\lambda\right).$
\end{proof}
\begin{theorem}
Let $\mathbf{Y}\sim\mathcal{MGAL}_{k\times n}\left(\mathbf{M},\mathbf{\Sigma},\mathbf{\Psi},\lambda\right)$, $\mathbf{X}\sim\mathcal{MN}_{k\times n}\left(\mathbf{0},\mathbf{\Sigma},\mathbf{\Psi}\right)$ and $W$ has the standart gamma distribution with shape parameter $\lambda$, independent of $\mathbf{X}$, then 
\begin{equation*}
\mathbf{Y}=\mathbf{M}W+W^{1/2}\mathbf{X}.
\end{equation*}
\end{theorem}
\begin{proof}
Let $\mathbf{Y}=\mathbf{M}W+W^{1/2}\mathbf{X}$.

Therefore,
\begin{equation*}
\operatorname{vec}\mathbf{Y}=\operatorname{vec}\left(\mathbf{M}W+W^{1/2}\mathbf{X}\right)=\left(\operatorname{vec}\mathbf{M}\right)W+W^{1/2}\operatorname{vec}\mathbf{X}.
\end{equation*}
By the definition of the matrix normal distribution \cite{gupta},
\begin{equation*}
\mathbf{X}\sim\mathcal{MN}_{k\times n}\left(\mathbf{M},\mathbf{\Sigma},\mathbf{\Psi}\right)\Leftrightarrow\operatorname{vec}\left(\mathbf{X}\right)\sim\mathcal{N}_{kn}\left(\operatorname{vec}\mathbf{M},\mathbf{\Psi}\otimes\mathbf{\Sigma}\right).
\end{equation*}
By the theorem \cite{kozub},
\begin{equation*}
\operatorname{vec}\left(\mathbf{Y}\right)\sim\mathcal{GAL}_{kn}\left(\operatorname{vec}\mathbf{M},\mathbf{\Psi}\otimes\mathbf{\Sigma},\lambda\right).
\end{equation*}
By Theorem \ref{vecaslapg},
\begin{equation*}
\mathbf{Y}\sim\mathcal{MGAL}_{k\times n}\left(\mathbf{M},\mathbf{\Sigma},\mathbf{\Psi},\lambda\right)\Leftrightarrow\operatorname{vec}\left(\mathbf{Y}\right)\sim\mathcal{GAL}_{kn}\left(\operatorname{vec}\mathbf{M},\mathbf{\Psi}\otimes\mathbf{\Sigma},\lambda\right).
\end{equation*}
\end{proof}
\begin{theorem}
\label{charmataslapg}
Let $\mathbf{T}\in \mathbb{R}^{k\times n}$. $\mathbf{X}\sim\mathcal{MGAL}_{k\times n}\left(\mathbf{M},\mathbf{\Sigma},\mathbf{\Psi},\lambda\right)$ if and only if the characteristic function of the random matrix $\mathbf{X}\in\mathbb{R}^{k\times n}$  is given by
\begin{equation*}
\varphi_{\mathbf{X}}\left(\mathbf{T}\right)=\left(\frac{1}{1+\frac{1}{2}\operatorname{tr}\left(\mathbf{\Psi}\mathbf{T}^{\mathrm{T}}\mathbf{\Sigma}\mathbf{T}\right)-\mathrm{i}\operatorname{tr}\left(\mathbf{M}^{\mathrm{T}}\mathbf{T}\right)}\right)^{\lambda}.
\end{equation*}
\end{theorem}
\begin{proof}
By Theorem \ref{vecaslapg},
\begin{equation*}
\mathbf{X}\sim\mathcal{MGAL}_{k\times n}\left(\mathbf{M},\mathbf{\Sigma},\mathbf{\Psi},\lambda\right)\Leftrightarrow\operatorname{vec}\left(\mathbf{X}\right)\sim\mathcal{GAL}_{kn}\left(\operatorname{vec}\mathbf{M},\mathbf{\Psi}\otimes\mathbf{\Sigma},\lambda\right).
\end{equation*}
By the definition of the multivariate generalized asymmetric Laplace distribution \cite{kozub},
\begin{equation*}
\varphi_{\operatorname{vec}\mathbf{X}}\left(\operatorname{vec}\mathbf{T}\right)=\left(\frac{1}{1+\frac{1}{2}\left(\operatorname{vec}\mathbf{T}\right)^{\mathrm{T}}\left(\mathbf{\Psi}\otimes\mathbf{\Sigma}\right)\left(\operatorname{vec}\mathbf{T}\right)-\mathrm{i}\left(\operatorname{vec}\mathbf{M}\right)^{\mathrm{T}}\left(\operatorname{vec}\mathbf{T}\right)}\right)^{\lambda}.
\end{equation*}
Thus,
\begin{equation*}
\varphi_{\operatorname{vec}\mathbf{X}}\left(\operatorname{vec}\mathbf{T}\right)=\varphi_{\mathbf{X}}\left(\mathbf{T}\right)=\left(\frac{1}{1+\frac{1}{2}\operatorname{tr}\left(\mathbf{\Psi}\mathbf{T}^{\mathrm{T}}\mathbf{\Sigma}\mathbf{T}\right)-\mathrm{i}\operatorname{tr}\left(\mathbf{M}^{\mathrm{T}}\mathbf{T}\right)}\right)^{\lambda}.
\end{equation*}
\end{proof}
\begin{theorem}
$\mathbf{X}\sim\mathcal{MGAL}_{k\times n}\left(\mathbf{M},\mathbf{\Sigma},\mathbf{\Psi},1\right)\Leftrightarrow\mathbf{X}\sim\mathcal{MAL}_{k\times n}\left(\mathbf{M},\mathbf{\Sigma},\mathbf{\Psi}\right).$
\begin{proof}
By Theorem \ref{charmataslapg}, $\mathbf{X}\sim\mathcal{MGAL}_{k\times n}\left(\mathbf{M},\mathbf{\Sigma},\mathbf{\Psi},1\right)$ if and only if
\begin{equation*}
\begin{split}
\varphi_{\mathbf{X}}\left(\mathbf{T}\right)&=\left(\frac{1}{1+\frac{1}{2}\operatorname{tr}\left(\mathbf{\Psi}\mathbf{T}^{\mathrm{T}}\mathbf{\Sigma}\mathbf{T}\right)-\mathrm{i}\operatorname{tr}\left(\mathbf{M}^{\mathrm{T}}\mathbf{T}\right)}\right)^{1}\\&=\frac{1}{1+\frac{1}{2}\operatorname{tr}\left(\mathbf{\Psi}\mathbf{T}^{\mathrm{T}}\mathbf{\Sigma}\mathbf{T}\right)-\mathrm{i}\operatorname{tr}\left(\mathbf{M}^{\mathrm{T}}\mathbf{T}\right)}.
\end{split}
\end{equation*}
By Theorem \ref{charmataslap}, $\mathbf{X}\sim\mathcal{MAL}_{k\times n}\left(\mathbf{M},\mathbf{\Sigma},\mathbf{\Psi}\right)$ if and only if
\begin{equation*}
\begin{split}
\varphi_{\mathbf{X}}\left(\mathbf{T}\right)=\frac{1}{1+\frac{1}{2}\operatorname{tr}\left(\mathbf{\Psi}\mathbf{T}^{\mathrm{T}}\mathbf{\Sigma}\mathbf{T}\right)-\mathrm{i}\operatorname{tr}\left(\mathbf{M}^{\mathrm{T}}\mathbf{T}\right)}.
\end{split}
\end{equation*}
Thus, $\mathbf{X}\sim\mathcal{MGAL}_{k\times n}\left(\mathbf{M},\mathbf{\Sigma},\mathbf{\Psi},1\right)\Leftrightarrow\mathbf{X}\sim\mathcal{MAL}_{k\times n}\left(\mathbf{M},\mathbf{\Sigma},\mathbf{\Psi}\right).$
\end{proof}
\end{theorem}
\begin{theorem}
\label{expvalueaslapg}
If $\mathbf{X}\sim\mathcal{MGAL}\left(\mathbf{M},\mathbf{\Sigma},\mathbf{\Psi},\lambda\right)$, then the expected value of the random matrix $\mathbf{X}\in\mathbb{R}^{k\times n}$ is given by
\begin{equation*}
 \mathbb{E}\left[\mathbf{X}\right]=\lambda\mathbf{M}.
\end{equation*}
\end{theorem}
\begin{proof}
By Theorem \ref{vecaslapg},
\begin{equation*}
\mathbf{X}\sim\mathcal{MGAL}_{k\times n}\left(\mathbf{M},\mathbf{\Sigma},\mathbf{\Psi},\lambda\right)\Leftrightarrow\operatorname{vec}\left(\mathbf{X}\right)\sim\mathcal{GAL}_{kn}\left(\operatorname{vec}\mathbf{M},\mathbf{\Psi}\otimes\mathbf{\Sigma},\lambda\right).
\end{equation*}
By the property of the multivariate generalized asymmetric Laplace distribution \cite{kozub},
\begin{equation*}
 \mathbb{E}\left[\operatorname{vec}\mathbf{X}\right]=\lambda\operatorname{vec}\mathbf{M}=\operatorname{vec}\left(\lambda\mathbf{M}\right).
\end{equation*}
Thus,
\begin{equation*}
 \mathbb{E}\left[\mathbf{X}\right]=\lambda\mathbf{M}.
\end{equation*}
\end{proof}
\begin{theorem}
\label{kxxaslapg}
If  $\mathbf{X}\sim\mathcal{MGAL}_{k\times n}\left(\mathbf{M},\mathbf{\Sigma},\mathbf{\Psi},\lambda\right)$, then the variance-covariance matrix of the random matrix $\mathbf{X}\in\mathbb{R}^{k\times n}$ is given by
\begin{equation*}
\mathbf{K}_{\mathbf{XX}}=\lambda\left(\mathbf{\Psi}\otimes\mathbf{\Sigma}+\left(\operatorname{vec}\mathbf{M}\right)\left(\operatorname{vec}\mathbf{M}\right)^{\mathrm{T}}\right).
\end{equation*}
\end{theorem}
\begin{proof}
By Theorem \ref{vecaslapg},
\begin{equation*}
\mathbf{X}\sim\mathcal{MGAL}_{k\times n}\left(\mathbf{M},\mathbf{\Sigma},\mathbf{\Psi},\lambda\right)\Leftrightarrow\operatorname{vec}\left(\mathbf{X}\right)\sim\mathcal{GAL}_{kn}\left(\operatorname{vec}\mathbf{M},\mathbf{\Psi}\otimes\mathbf{\Sigma},\lambda\right).
\end{equation*}
By the property of the multivariate asymmetric generalized Laplace distribution \cite{kozub},
\begin{equation*}
\mathbf{K}_{\operatorname{vec}\mathbf{X},\operatorname{vec}\mathbf{X}}=\lambda\left(\mathbf{\Psi}\otimes\mathbf{\Sigma}+\left(\operatorname{vec}\mathbf{M}\right)\left(\operatorname{vec}\mathbf{M}\right)^{\mathrm{T}}\right).
\end{equation*}
Thus,
\begin{equation*}
\mathbf{K}_{\mathbf{XX}}=\lambda\left(\mathbf{\Psi}\otimes\mathbf{\Sigma}+\left(\operatorname{vec}\mathbf{M}\right)\left(\operatorname{vec}\mathbf{M}\right)^{\mathrm{T}}\right).
\end{equation*}
\end{proof}
\begin{theorem}
Let $\mathbf{D}$ is a rectangular $m\times k$ matrix, $\mathbf{C}$ is a rectangular $n\times p$ matrix, $m\leq k$ and $p\leq n$, then
\begin{equation*}
\mathbf{X}\sim\mathcal{MGAL}_{k\times n}\left(\mathbf{M},\mathbf{\Sigma},\mathbf{\Psi},\lambda\right)\Leftrightarrow\mathbf{DXC}\sim\mathcal{MGAL}_{m\times p}\left(\mathbf{DMC},\mathbf{D\Sigma}\mathbf{D}^{\mathrm{T}},\mathbf{C}^{\mathrm{T}}\mathbf{\Psi C},\lambda\right).
\end{equation*}
\end{theorem}
\begin{proof}
The characteristic function of $\mathbf{DXC}$ is
\begin{equation*}
\varphi_{\mathbf{DXC}}\left(\mathbf{T}\right)=\mathbb{E}\left[e^{\mathrm{i}\operatorname{tr}\left(\mathbf{DXC}\mathbf{T}^{\mathrm{T}}\right)}\right]=\mathbb{E}\left[e^{\mathrm{i}\operatorname{tr}\left(\mathbf{X}\mathbf{C}\mathbf{T}^{\mathrm{T}}\mathbf{D}\right)}\right].
\end{equation*}
By Theorem \ref{charmataslapg},
\begin{equation*}
\begin{split}
\varphi_{\mathbf{DXC}}\left(\mathbf{T}\right)&=\left(\frac{1}{1+\frac{1}{2}\operatorname{tr}\left(\mathbf{\Psi}\mathbf{C}\mathbf{T}^{\mathrm{T}}\mathbf{D}\mathbf{\Sigma}\mathbf{D}^{\mathrm{T}}\mathbf{T}\mathbf{C}^{\mathrm{T}}\right)-\mathrm{i}\operatorname{tr}\left(\mathbf{C}\mathbf{T}^{\mathrm{T}}\mathbf{D}\mathbf{M}\right)}\right)^{\lambda}\\&=\left(\frac{1}{1+\frac{1}{2}\operatorname{tr}\left(\mathbf{C}^{\mathrm{T}}\mathbf{\Psi}\mathbf{C}\mathbf{T}^{\mathrm{T}}\mathbf{D}\mathbf{\Sigma}\mathbf{D}^{\mathrm{T}}\mathbf{T}\right)-\mathrm{i}\operatorname{tr}\left(\mathbf{T}^{\mathrm{T}}\mathbf{D}\mathbf{M}\mathbf{C}\right)}\right)^{\lambda}.
\end{split}
\end{equation*}
Thus,
\begin{equation*}
\mathbf{X}\sim\mathcal{MGAL}_{k\times n}\left(\mathbf{M},\mathbf{\Sigma},\mathbf{\Psi},\lambda\right)\Leftrightarrow\mathbf{DXC}\sim\mathcal{MGAL}_{m\times p}\left(\mathbf{DMC},\mathbf{D\Sigma}\mathbf{D}^{\mathrm{T}},\mathbf{C}^{\mathrm{T}}\mathbf{\Psi C},\lambda\right).
\end{equation*}
\end{proof}
\begin{cor}
Let $\mathbf{D}$ is a rectangular $m\times k$ matrix and $m\leq k$, then
\begin{equation*}
\mathbf{X}\sim\mathcal{MGAL}_{k\times n}\left(\mathbf{M},\mathbf{\Sigma},\mathbf{\Psi},\lambda\right)\Leftrightarrow\mathbf{DX}\sim\mathcal{MGAL}_{m\times n}\left(\mathbf{DM},\mathbf{D\Sigma}\mathbf{D}^{\mathrm{T}},\mathbf{\Psi},\lambda\right).
\end{equation*}
\end{cor}
\begin{cor}
Let $\mathbf{C}$ is a rectangular $n\times p$ matrix and $p\leq n$, then
\begin{equation*}
\mathbf{X}\sim\mathcal{MGAL}_{k\times n}\left(\mathbf{M},\mathbf{\Sigma},\mathbf{\Psi},\lambda\right)\Leftrightarrow\mathbf{XC}\sim\mathcal{MGAL}_{k\times p}\left(\mathbf{MC},\mathbf{\Sigma},\mathbf{C}^{\mathrm{T}}\mathbf{\Psi C},\lambda\right).
\end{equation*}
\end{cor}
\section{Tensor variate Laplace distributions}
\subsection{Introduction}
The tensor variate asymmetric Laplace distribution is a continuous probability distribution that is a generalization of the matrix variate asymmetric Laplace distribution to tensor-valued random variables. The tensor variate asymmetric Laplace distribution of a random order-D tensor $\boldsymbol{\mathcal{X}}$, with  dimensional lengths $n_1\times n_2 \times \cdots \times n_D=\mathbf{n}$ can be written in the following notation:
\begin{equation*}
\mathcal{TAL}_{\mathbf{n}}\left(\boldsymbol{\mathfrak{M}},\mathbf{\Sigma}_1,\mathbf{\Sigma}_2,\ldots,\mathbf{\Sigma}_D\right)
\end{equation*}
or
\begin{equation*}
\mathcal{TAL}\left(\boldsymbol{\mathfrak{M}},\mathbf{\Sigma}_1,\mathbf{\Sigma}_2,\ldots,\mathbf{\Sigma}_D\right),
\end{equation*}
where $\boldsymbol{\mathfrak{M}}$ is a location tensor, $\mathbf{\Sigma}_i$ is a positive-definite scale matrix.

The tensor variate generalized asymmetric Laplace distribution is a continuous probability distribution that is a generalization of the matrix variate generalized asymmetric Laplace distribution to tensor-valued random variables. The tensor variate generalized asymmetric Laplace distribution of a random order-D tensor $\boldsymbol{\mathcal{X}}$, with  dimensional lengths $n_1\times n_2 \times \cdots \times n_D=\mathbf{n}$ can be written in the following notation:
\begin{equation*}
\mathcal{TGAL}_{\mathbf{n}}\left(\boldsymbol{\mathfrak{M}},\mathbf{\Sigma}_1,\mathbf{\Sigma}_2,\ldots,\mathbf{\Sigma}_D,\lambda\right)
\end{equation*}
or
\begin{equation*}
\mathcal{TGAL}\left(\boldsymbol{\mathfrak{M}},\mathbf{\Sigma}_1,\mathbf{\Sigma}_2,\ldots,\mathbf{\Sigma}_D,\lambda\right),
\end{equation*}
where $\boldsymbol{\mathfrak{M}}$ is a location tensor, $\mathbf{\Sigma}_i$ is a positive-definite scale matrix and $\lambda>0$.
\subsection{Tensor variate asymmetric Laplace distribution}
\begin{defi}
\label{pmfmataslap}
Let  $\boldsymbol{\mathcal{X}}$ is a random order-D tensor, with dimensional lengths $n_1\times n_2 \times \cdots \times n_D$, with realization $\boldsymbol{\mathfrak{X}}$. $\boldsymbol{\mathcal{X}}\sim\mathcal{TAL}\left(\boldsymbol{\mathfrak{M}},\mathbf{\Sigma}_1,\mathbf{\Sigma}_2,\ldots,\mathbf{\Sigma}_D\right)$ if and only if the probability density function is given by
\begin{equation*}
\begin{split}
f_{\boldsymbol{\mathcal{X}}}\left(\boldsymbol{\mathfrak{X}}\right)&=\frac{2\operatorname{exp}\left(\left(\operatorname{vec}\boldsymbol{\mathfrak{X}}\right)^{\mathrm{T}}\left(\bigotimes_{i=1}^{D}\mathbf{\Sigma}_i\right)^{-1}\operatorname{vec}\boldsymbol{\mathfrak{M}}\right)}{\left(2\pi\right)^{n^{*}/2}\prod_{i=1}^{D}\left(\det\mathbf{\Sigma}_{i}\right)^{n^{*}/(2n_i)}}\left(\frac{\left(\operatorname{vec}\boldsymbol{\mathfrak{X}}\right)^{\mathrm{T}}\left(\bigotimes_{i=1}^{D}\mathbf{\Sigma}_i\right)^{-1}\operatorname{vec}\boldsymbol{\mathfrak{X}}}{2+\left(\operatorname{vec}\boldsymbol{\mathfrak{M}}\right)^{\mathrm{T}}\left(\bigotimes_{i=1}^{D}\mathbf{\Sigma}_i\right)^{-1}\operatorname{vec}\boldsymbol{\mathfrak{M}}}\right)^{1/2-n^{*}/4}\\&\times K_{1-n^{*}/2}\left(\sqrt{\left(2+\left(\operatorname{vec}\boldsymbol{\mathfrak{M}}\right)^{\mathrm{T}}\left(\bigotimes_{i=1}^{D}\mathbf{\Sigma}_i\right)^{-1}\operatorname{vec}\boldsymbol{\mathfrak{M}}\right)\left(\left(\operatorname{vec}\boldsymbol{\mathfrak{X}}\right)^{\mathrm{T}}\left(\bigotimes_{i=1}^{D}\mathbf{\Sigma}_i\right)^{-1}\operatorname{vec}\boldsymbol{\mathfrak{X}}\right)}\right),
\end{split}
\end{equation*}
where $n^{*}=\prod_{i=1}^{D}n_i$ and $K_{1-n^{*}/2}(\cdot)$ is the modified Bessel function of the third kind.
\end{defi}
\begin{theorem}
\label{vecaslap}
Let  $\boldsymbol{\mathcal{X}}$ is a random order-D tensor, with dimensional lengths $n_1\times n_2 \times \cdots \times n_D$, with realization $\boldsymbol{\mathfrak{X}}$. $\boldsymbol{\mathcal{X}}\sim\mathcal{TAL}_{\mathbf{n}}\left(\boldsymbol{\mathfrak{M}},\mathbf{\Sigma}_1,\mathbf{\Sigma}_2,\ldots,\mathbf{\Sigma}_D\right)\Leftrightarrow\operatorname{vec}\boldsymbol{\mathcal{X}}\sim\mathcal{AL}_{n^{*}}\left(\operatorname{vec}\boldsymbol{\mathfrak{M}},\bigotimes_{i=1}^{D}\mathbf{\Sigma}_i\right).$
\end{theorem}
\begin{proof}
\begin{equation*}
\begin{split}
f_{\boldsymbol{\mathcal{X}}}\left(\boldsymbol{\mathfrak{X}}\right)&=\frac{2\operatorname{exp}\left(\left(\operatorname{vec}\boldsymbol{\mathfrak{X}}\right)^{\mathrm{T}}\left(\bigotimes_{i=1}^{D}\mathbf{\Sigma}_i\right)^{-1}\operatorname{vec}\boldsymbol{\mathfrak{M}}\right)}{\left(2\pi\right)^{n^{*}/2}\prod_{i=1}^{D}\left(\det\mathbf{\Sigma}_{i}\right)^{n^{*}/(2n_i)}}\left(\frac{\left(\operatorname{vec}\boldsymbol{\mathfrak{X}}\right)^{\mathrm{T}}\left(\bigotimes_{i=1}^{D}\mathbf{\Sigma}_i\right)^{-1}\operatorname{vec}\boldsymbol{\mathfrak{X}}}{2+\left(\operatorname{vec}\boldsymbol{\mathfrak{M}}\right)^{\mathrm{T}}\left(\bigotimes_{i=1}^{D}\mathbf{\Sigma}_i\right)^{-1}\operatorname{vec}\boldsymbol{\mathfrak{M}}}\right)^{1/2-n^{*}/4}\\&\times K_{1-n^{*}/2}\left(\sqrt{\left(2+\left(\operatorname{vec}\boldsymbol{\mathfrak{M}}\right)^{\mathrm{T}}\left(\bigotimes_{i=1}^{D}\mathbf{\Sigma}_i\right)^{-1}\operatorname{vec}\boldsymbol{\mathfrak{M}}\right)\left(\left(\operatorname{vec}\boldsymbol{\mathfrak{X}}\right)^{\mathrm{T}}\left(\bigotimes_{i=1}^{D}\mathbf{\Sigma}_i\right)^{-1}\operatorname{vec}\boldsymbol{\mathfrak{X}}\right)}\right),
\end{split}
\end{equation*}
where $n^{*}=\prod_{i=1}^{D}n_i$ and $K_{1-n^{*}/2}(\cdot)$ is the modified Bessel function of the third kind.

By the definition of the multivariate asymmetric Laplace distribution \cite{kotzas}, $\operatorname{vec}\boldsymbol{\mathcal{X}}\sim\mathcal{AL}_{n^{*}}\left(\operatorname{vec}\boldsymbol{\mathfrak{M}},\bigotimes_{i=1}^{D}\mathbf{\Sigma}_i\right)$ if and only if
\begin{equation*}
\begin{split}
f_{\operatorname{vec}\boldsymbol{\mathcal{X}}}\left(\operatorname{vec}\boldsymbol{\mathfrak{X}}\right)&=\frac{2\operatorname{exp}\left(\left(\operatorname{vec}\boldsymbol{\mathfrak{X}}\right)^{\mathrm{T}}\left(\bigotimes_{i=1}^{D}\mathbf{\Sigma}_i\right)^{-1}\operatorname{vec}\boldsymbol{\mathfrak{M}}\right)}{\left(2\pi\right)^{n^{*}/2}\prod_{i=1}^{D}\left(\det\mathbf{\Sigma}_{i}\right)^{n^{*}/(2n_i)}}\left(\frac{\left(\operatorname{vec}\boldsymbol{\mathfrak{X}}\right)^{\mathrm{T}}\left(\bigotimes_{i=1}^{D}\mathbf{\Sigma}_i\right)^{-1}\operatorname{vec}\boldsymbol{\mathfrak{X}}}{2+\left(\operatorname{vec}\boldsymbol{\mathfrak{M}}\right)^{\mathrm{T}}\left(\bigotimes_{i=1}^{D}\mathbf{\Sigma}_i\right)^{-1}\operatorname{vec}\boldsymbol{\mathfrak{M}}}\right)^{1/2-n^{*}/4}\\&\times K_{1-n^{*}/2}\left(\sqrt{\left(2+\left(\operatorname{vec}\boldsymbol{\mathfrak{M}}\right)^{\mathrm{T}}\left(\bigotimes_{i=1}^{D}\mathbf{\Sigma}_i\right)^{-1}\operatorname{vec}\boldsymbol{\mathfrak{M}}\right)\left(\left(\operatorname{vec}\boldsymbol{\mathfrak{X}}\right)^{\mathrm{T}}\left(\bigotimes_{i=1}^{D}\mathbf{\Sigma}_i\right)^{-1}\operatorname{vec}\boldsymbol{\mathfrak{X}}\right)}\right),
\end{split}
\end{equation*}
where $n^{*}=\prod_{i=1}^{D}n_i$ and $K_{1-n^{*}/2}(\cdot)$ is the modified Bessel function of the third kind.

Thus, $\boldsymbol{\mathcal{X}}\sim\mathcal{TAL}_{\mathbf{n}}\left(\boldsymbol{\mathfrak{M}},\mathbf{\Sigma}_1,\mathbf{\Sigma}_2,\ldots,\mathbf{\Sigma}_D\right)\Leftrightarrow\operatorname{vec}\boldsymbol{\mathcal{X}}\sim\mathcal{AL}_{n^{*}}\left(\operatorname{vec}\boldsymbol{\mathfrak{M}},\bigotimes_{i=1}^{D}\mathbf{\Sigma}_i\right).$
\end{proof}
\begin{theorem}
\label{charmataslap}
Let $\boldsymbol{\mathfrak{T}}$ is an $\mathbf{n}$ dimensional order-D tensor. $\boldsymbol{\mathcal{X}}\sim\mathcal{TAL}_{\mathbf{n}}\left(\boldsymbol{\mathfrak{M}},\mathbf{\Sigma}_1,\mathbf{\Sigma}_2,\ldots,\mathbf{\Sigma}_D\right)$ if and only if the characteristic function of the random tensor is given by
\begin{equation*}
\varphi_{\boldsymbol{\mathcal{X}}}\left(\boldsymbol{\mathfrak{T}}\right)=\frac{1}{1+\frac{1}{2}\left(\operatorname{vec}\boldsymbol{\mathfrak{T}}\right)^{\mathrm{T}}\left(\bigotimes_{i=1}^{D}\mathbf{\Sigma}_i\right)\operatorname{vec}\boldsymbol{\mathfrak{T}}-\mathrm{i}\left(\operatorname{vec}\boldsymbol{\mathfrak{M}}\right)^{\mathrm{T}}\operatorname{vec}\boldsymbol{\mathfrak{T}}}.
\end{equation*}
\end{theorem}
\begin{proof}
By Theorem \ref{vecaslap},
\begin{equation*}
\boldsymbol{\mathcal{X}}\sim\mathcal{TAL}_{\mathbf{n}}\left(\boldsymbol{\mathfrak{M}},\mathbf{\Sigma}_1,\mathbf{\Sigma}_2,\ldots,\mathbf{\Sigma}_D\right)\Leftrightarrow\operatorname{vec}\boldsymbol{\mathcal{X}}\sim\mathcal{AL}_{n^{*}}\left(\operatorname{vec}\boldsymbol{\mathfrak{M}},\bigotimes_{i=1}^{D}\mathbf{\Sigma}_i\right).
\end{equation*}
By the definition of the multivariate asymmetric Laplace distribution \cite{kotzas}, $\operatorname{vec}\boldsymbol{\mathcal{X}}\sim\mathcal{AL}_{n^{*}}\left(\operatorname{vec}\boldsymbol{\mathfrak{M}},\bigotimes_{i=1}^{D}\mathbf{\Sigma}_i\right)$ if and only if
\begin{equation*}
\varphi_{\operatorname{vec}\boldsymbol{\mathcal{X}}}\left(\operatorname{vec}\boldsymbol{\mathfrak{T}}\right)=\frac{1}{1+\frac{1}{2}\left(\operatorname{vec}\boldsymbol{\mathfrak{T}}\right)^{\mathrm{T}}\left(\bigotimes_{i=1}^{D}\mathbf{\Sigma}_i\right)\operatorname{vec}\boldsymbol{\mathfrak{T}}-\mathrm{i}\left(\operatorname{vec}\boldsymbol{\mathfrak{M}}\right)^{\mathrm{T}}\operatorname{vec}\boldsymbol{\mathfrak{T}}}.
\end{equation*}
Thus,
\begin{equation*}
\varphi_{\operatorname{vec}\boldsymbol{\mathcal{X}}}\left(\operatorname{vec}\boldsymbol{\mathfrak{T}}\right)=\varphi_{\boldsymbol{\mathcal{X}}}\left(\boldsymbol{\mathfrak{T}}\right)=\frac{1}{1+\frac{1}{2}\left(\operatorname{vec}\boldsymbol{\mathfrak{T}}\right)^{\mathrm{T}}\left(\bigotimes_{i=1}^{D}\mathbf{\Sigma}_i\right)\operatorname{vec}\boldsymbol{\mathfrak{T}}-\mathrm{i}\left(\operatorname{vec}\boldsymbol{\mathfrak{M}}\right)^{\mathrm{T}}\operatorname{vec}\boldsymbol{\mathfrak{T}}}.
\end{equation*}
\end{proof}
\begin{theorem}
Let $\boldsymbol{\mathcal{Y}}\sim\mathcal{TAL}_{\mathbf{n}}\left(\boldsymbol{\mathfrak{M}},\mathbf{\Sigma}_1,\mathbf{\Sigma}_2,\ldots,\mathbf{\Sigma}_D\right)$, $\boldsymbol{\mathcal{X}}\sim\mathcal{TN}_{\mathbf{n}}\left(\mathbb{O},\mathbf{\Sigma}_1,\mathbf{\Sigma}_2,\ldots,\mathbf{\Sigma}_D\right)$ and $W\sim\mathrm{Exp}\left(1\right)$, independent of $\boldsymbol{\mathcal{X}}$, then 
\begin{equation*}
\boldsymbol{\mathcal{Y}}=\boldsymbol{\mathfrak{M}}W+W^{1/2}\boldsymbol{\mathcal{X}}.
\end{equation*}
\end{theorem}
\begin{proof}
Let $\boldsymbol{\mathcal{Y}}=\boldsymbol{\mathfrak{M}}W+W^{1/2}\boldsymbol{\mathcal{X}}$.

Therefore,
\begin{equation*}
\operatorname{vec}\boldsymbol{\mathcal{Y}}=\operatorname{vec}\left(\boldsymbol{\mathfrak{M}}W+W^{1/2}\boldsymbol{\mathcal{X}}\right)=\left(\operatorname{vec}\boldsymbol{\mathfrak{M}}\right)W+W^{1/2}\operatorname{vec}\boldsymbol{\mathcal{X}}.
\end{equation*}
By the definition of the tensor normal distribution \cite[Theorem 2.1]{galtensor},
\begin{equation*}
\boldsymbol{\mathcal{X}}\sim\mathcal{TN}_{\mathbf{n}}\left(\boldsymbol{\mathfrak{M}},\mathbf{\Sigma}_1,\mathbf{\Sigma}_2,\ldots,\mathbf{\Sigma}_D\right)\Leftrightarrow\operatorname{vec}\boldsymbol{\mathcal{X}}\sim\mathcal{N}_{n^{*}}\left(\operatorname{vec}\boldsymbol{\mathfrak{M}},\bigotimes_{i=1}^{D}\mathbf{\Sigma}_i\right).
\end{equation*}
By the theorem \cite[Theorem 6.3.1]{kotzas},
\begin{equation*}
\operatorname{vec}\boldsymbol{\mathcal{Y}}\sim\mathcal{AL}_{n^{*}}\left(\operatorname{vec}\boldsymbol{\mathfrak{M}},\bigotimes_{i=1}^{D}\mathbf{\Sigma}_i\right).
\end{equation*}
By Theorem \ref{vecaslap},
\begin{equation*}
\operatorname{vec}\boldsymbol{\mathcal{Y}}\sim\mathcal{AL}_{n^{*}}\left(\operatorname{vec}\boldsymbol{\mathfrak{M}},\bigotimes_{i=1}^{D}\mathbf{\Sigma}_i\right)\Leftrightarrow \boldsymbol{\mathcal{Y}}\sim\mathcal{TAL}_{\mathbf{n}}\left(\boldsymbol{\mathfrak{M}},\mathbf{\Sigma}_1,\mathbf{\Sigma}_2,\ldots,\mathbf{\Sigma}_D\right).
\end{equation*}
\end{proof}
\begin{theorem}
\label{aslapvarg}
Let  $\boldsymbol{\mathcal{X}}$ is a random order-D tensor, with dimensional lengths $n_1\times n_2 \times \cdots \times n_D$, with realization $\boldsymbol{\mathfrak{X}}$. $\boldsymbol{\mathcal{X}}\sim\mathcal{TAL}_{\mathbf{n}}\left(\boldsymbol{\mathfrak{M}},\mathbf{\Sigma}_1,\mathbf{\Sigma}_2,\ldots,\mathbf{\Sigma}_D\right)\Leftrightarrow\boldsymbol{\mathcal{X}}\sim\mathcal{TVVG}_{\mathbf{n}}\left(\mathbb{O},\boldsymbol{\mathfrak{M}},\mathbf{\Sigma}_1,\mathbf{\Sigma}_2,\ldots,\mathbf{\Sigma}_D,1\right).$
\end{theorem}
\begin{proof}
By Definition \ref{pmfmataslap}, $\boldsymbol{\mathcal{X}}\sim\mathcal{TAL}\left(\boldsymbol{\mathfrak{M}},\mathbf{\Sigma}_1,\mathbf{\Sigma}_2,\ldots,\mathbf{\Sigma}_D\right)$ if and only if the probability density function is given by
\begin{equation*}
\begin{split}
f_{\boldsymbol{\mathcal{X}}}\left(\boldsymbol{\mathfrak{X}}\right)&=\frac{2\operatorname{exp}\left(\left(\operatorname{vec}\boldsymbol{\mathfrak{X}}\right)^{\mathrm{T}}\left(\bigotimes_{i=1}^{D}\mathbf{\Sigma}_i\right)^{-1}\operatorname{vec}\boldsymbol{\mathfrak{M}}\right)}{\left(2\pi\right)^{n^{*}/2}\prod_{i=1}^{D}\left(\det\mathbf{\Sigma}_{i}\right)^{n^{*}/(2n_i)}}\left(\frac{\left(\operatorname{vec}\boldsymbol{\mathfrak{X}}\right)^{\mathrm{T}}\left(\bigotimes_{i=1}^{D}\mathbf{\Sigma}_i\right)^{-1}\operatorname{vec}\boldsymbol{\mathfrak{X}}}{2+\left(\operatorname{vec}\boldsymbol{\mathfrak{M}}\right)^{\mathrm{T}}\left(\bigotimes_{i=1}^{D}\mathbf{\Sigma}_i\right)^{-1}\operatorname{vec}\boldsymbol{\mathfrak{M}}}\right)^{1/2-n^{*}/4}\\&\times K_{1-n^{*}/2}\left(\sqrt{\left(2+\left(\operatorname{vec}\boldsymbol{\mathfrak{M}}\right)^{\mathrm{T}}\left(\bigotimes_{i=1}^{D}\mathbf{\Sigma}_i\right)^{-1}\operatorname{vec}\boldsymbol{\mathfrak{M}}\right)\left(\left(\operatorname{vec}\boldsymbol{\mathfrak{X}}\right)^{\mathrm{T}}\left(\bigotimes_{i=1}^{D}\mathbf{\Sigma}_i\right)^{-1}\operatorname{vec}\boldsymbol{\mathfrak{X}}\right)}\right),
\end{split}
\end{equation*}
where $n^{*}=\prod_{i=1}^{D}n_i$ and $K_{1-n^{*}/2}(\cdot)$ is the modified Bessel function of the third kind.

By the definition of the tensor variance gamma distribution \cite{galtensor}, $\boldsymbol{\mathcal{X}}\sim\mathcal{TVVG}_{\mathbf{n}}\left(\mathbb{O},\boldsymbol{\mathfrak{M}},\mathbf{\Sigma}_1,\mathbf{\Sigma}_2,\ldots,\mathbf{\Sigma}_D,1\right)$ if and only if 
\begin{equation*}
\begin{split}
f_{\boldsymbol{\mathcal{X}}}\left(\boldsymbol{\mathfrak{X}}\right)&=\frac{2\operatorname{exp}\left(\left(\operatorname{vec}\boldsymbol{\mathfrak{X}}\right)^{\mathrm{T}}\left(\bigotimes_{i=1}^{D}\mathbf{\Sigma}_i\right)^{-1}\operatorname{vec}\boldsymbol{\mathfrak{M}}\right)}{\left(2\pi\right)^{n^{*}/2}\prod_{i=1}^{D}\left(\det\mathbf{\Sigma}_{i}\right)^{n^{*}/(2n_i)}}\left(\frac{\left(\operatorname{vec}\boldsymbol{\mathfrak{X}}\right)^{\mathrm{T}}\left(\bigotimes_{i=1}^{D}\mathbf{\Sigma}_i\right)^{-1}\operatorname{vec}\boldsymbol{\mathfrak{X}}}{2+\left(\operatorname{vec}\boldsymbol{\mathfrak{M}}\right)^{\mathrm{T}}\left(\bigotimes_{i=1}^{D}\mathbf{\Sigma}_i\right)^{-1}\operatorname{vec}\boldsymbol{\mathfrak{M}}}\right)^{1/2-n^{*}/4}\\&\times K_{1-n^{*}/2}\left(\sqrt{\left(2+\left(\operatorname{vec}\boldsymbol{\mathfrak{M}}\right)^{\mathrm{T}}\left(\bigotimes_{i=1}^{D}\mathbf{\Sigma}_i\right)^{-1}\operatorname{vec}\boldsymbol{\mathfrak{M}}\right)\left(\left(\operatorname{vec}\boldsymbol{\mathfrak{X}}\right)^{\mathrm{T}}\left(\bigotimes_{i=1}^{D}\mathbf{\Sigma}_i\right)^{-1}\operatorname{vec}\boldsymbol{\mathfrak{X}}\right)}\right),
\end{split}
\end{equation*}
where $n^{*}=\prod_{i=1}^{D}n_i$ and $K_{1-n^{*}/2}(\cdot)$ is the modified Bessel function of the third kind.

Thus, $\boldsymbol{\mathcal{X}}\sim\mathcal{TAL}_{\mathbf{n}}\left(\boldsymbol{\mathfrak{M}},\mathbf{\Sigma}_1,\mathbf{\Sigma}_2,\ldots,\mathbf{\Sigma}_D\right)\Leftrightarrow\boldsymbol{\mathcal{X}}\sim\mathcal{TVVG}_{\mathbf{n}}\left(\mathbb{O},\boldsymbol{\mathfrak{M}},\mathbf{\Sigma}_1,\mathbf{\Sigma}_2,\ldots,\mathbf{\Sigma}_D,1\right).$
\end{proof}
\begin{theorem}
\label{expvalueaslap}
If $\boldsymbol{\mathcal{X}}\sim\mathcal{TAL}\left(\boldsymbol{\mathfrak{M}},\mathbf{\Sigma}_1,\mathbf{\Sigma}_2,\ldots,\mathbf{\Sigma}_D\right)$, then the expected value of the random tensor $\boldsymbol{\mathcal{X}}$ is given by
\begin{equation*}
 \mathbb{E}\left[\boldsymbol{\mathcal{X}}\right]=\boldsymbol{\mathfrak{M}}.
\end{equation*}
\end{theorem}
\begin{proof}
By Theorem \ref{vecaslap},
\begin{equation*}
\boldsymbol{\mathcal{X}}\sim\mathcal{TAL}_{\mathbf{n}}\left(\boldsymbol{\mathfrak{M}},\mathbf{\Sigma}_1,\mathbf{\Sigma}_2,\ldots,\mathbf{\Sigma}_D\right)\Leftrightarrow\operatorname{vec}\boldsymbol{\mathcal{X}}\sim\mathcal{AL}_{n^{*}}\left(\operatorname{vec}\boldsymbol{\mathfrak{M}},\bigotimes_{i=1}^{D}\mathbf{\Sigma}_i\right).
\end{equation*}
By the property of the multivariate asymmetric Laplace distribution \cite{kotzas},
\begin{equation*}
 \mathbb{E}\left[\operatorname{vec}\boldsymbol{\mathcal{X}}\right]=\operatorname{vec}\boldsymbol{\mathfrak{M}}.
\end{equation*}
Thus,
\begin{equation*}
 \mathbb{E}\left[\boldsymbol{\mathcal{X}}\right]=\boldsymbol{\mathfrak{M}}.
\end{equation*}
\end{proof}
\begin{theorem}
\label{kxxaslap}
If  $\boldsymbol{\mathcal{X}}\sim\mathcal{TAL}_{\mathbf{n}}\left(\boldsymbol{\mathfrak{M}},\mathbf{\Sigma}_1,\mathbf{\Sigma}_2,\ldots,\mathbf{\Sigma}_D\right)$, then the variance-covariance matrix of the random tensor $\boldsymbol{\mathcal{X}}$ is given by
\begin{equation*}
\mathbf{K}_{\boldsymbol{\mathcal{XX}}}=\bigotimes_{i=1}^{D}\mathbf{\Sigma}_i+\left(\operatorname{vec}\boldsymbol{\mathfrak{M}}\right)\left(\operatorname{vec}\boldsymbol{\mathfrak{M}}\right)^{\mathrm{T}}.
\end{equation*}
\end{theorem}
\begin{proof}
By Theorem \ref{vecaslap},
\begin{equation*}
\boldsymbol{\mathcal{X}}\sim\mathcal{TAL}_{\mathbf{n}}\left(\boldsymbol{\mathfrak{M}},\mathbf{\Sigma}_1,\mathbf{\Sigma}_2,\ldots,\mathbf{\Sigma}_D\right)\Leftrightarrow\operatorname{vec}\boldsymbol{\mathcal{X}}\sim\mathcal{AL}_{n^{*}}\left(\operatorname{vec}\boldsymbol{\mathfrak{M}},\bigotimes_{i=1}^{D}\mathbf{\Sigma}_i\right).
\end{equation*}
By the property of the multivariate asymmetric Laplace distribution \cite{kotzas}, if $\operatorname{vec}\boldsymbol{\mathcal{X}}\sim\mathcal{AL}_{n^{*}}\left(\operatorname{vec}\boldsymbol{\mathfrak{M}},\bigotimes_{i=1}^{D}\mathbf{\Sigma}_i\right)$, then
\begin{equation*}
\mathbf{K}_{\operatorname{vec}\boldsymbol{\mathcal{X}},\operatorname{vec}\boldsymbol{\mathcal{X}}}=\bigotimes_{i=1}^{D}\mathbf{\Sigma}_i+\left(\operatorname{vec}\boldsymbol{\mathfrak{M}}\right)\left(\operatorname{vec}\boldsymbol{\mathfrak{M}}\right)^{\mathrm{T}}.
\end{equation*}
Thus,
\begin{equation*}
\mathbf{K}_{\boldsymbol{\mathcal{XX}}}=\mathbf{K}_{\operatorname{vec}\boldsymbol{\mathcal{X}},\operatorname{vec}\boldsymbol{\mathcal{X}}}=\bigotimes_{i=1}^{D}\mathbf{\Sigma}_i+\left(\operatorname{vec}\boldsymbol{\mathfrak{M}}\right)\left(\operatorname{vec}\boldsymbol{\mathfrak{M}}\right)^{\mathrm{T}}.
\end{equation*}
\end{proof}
\subsection{Tensor variate generalized asymmetric Laplace distribution}
\begin{defi}
\label{pmfmataslap}
Let  $\boldsymbol{\mathcal{X}}$ is a random order-D tensor, with dimensional lengths $n_1\times n_2 \times \cdots \times n_D$, with realization $\boldsymbol{\mathfrak{X}}$. $\boldsymbol{\mathcal{X}}\sim\mathcal{TGAL}\left(\boldsymbol{\mathfrak{M}},\mathbf{\Sigma}_1,\mathbf{\Sigma}_2,\ldots,\mathbf{\Sigma}_D,\lambda\right)$ if and only if the probability density function is given by
\begin{equation*}
\begin{split}
f_{\boldsymbol{\mathcal{X}}}\left(\boldsymbol{\mathfrak{X}}\right)&=\frac{2\operatorname{exp}\left(\left(\operatorname{vec}\boldsymbol{\mathfrak{X}}\right)^{\mathrm{T}}\left(\bigotimes_{i=1}^{D}\mathbf{\Sigma}_i\right)^{-1}\operatorname{vec}\boldsymbol{\mathfrak{M}}\right)}{\left(2\pi\right)^{n^{*}/2}\Gamma\left(\lambda\right)\prod_{i=1}^{D}\left(\det\mathbf{\Sigma}_{i}\right)^{n^{*}/(2n_i)}}\left(\frac{\left(\operatorname{vec}\boldsymbol{\mathfrak{X}}\right)^{\mathrm{T}}\left(\bigotimes_{i=1}^{D}\mathbf{\Sigma}_i\right)^{-1}\operatorname{vec}\boldsymbol{\mathfrak{X}}}{2+\left(\operatorname{vec}\boldsymbol{\mathfrak{M}}\right)^{\mathrm{T}}\left(\bigotimes_{i=1}^{D}\mathbf{\Sigma}_i\right)^{-1}\operatorname{vec}\boldsymbol{\mathfrak{M}}}\right)^{\lambda/2-n^{*}/4}\\&\times K_{\lambda-n^{*}/2}\left(\sqrt{\left(2+\left(\operatorname{vec}\boldsymbol{\mathfrak{M}}\right)^{\mathrm{T}}\left(\bigotimes_{i=1}^{D}\mathbf{\Sigma}_i\right)^{-1}\operatorname{vec}\boldsymbol{\mathfrak{M}}\right)\left(\left(\operatorname{vec}\boldsymbol{\mathfrak{X}}\right)^{\mathrm{T}}\left(\bigotimes_{i=1}^{D}\mathbf{\Sigma}_i\right)^{-1}\operatorname{vec}\boldsymbol{\mathfrak{X}}\right)}\right),
\end{split}
\end{equation*}
where $n^{*}=\prod_{i=1}^{D}n_i$ and $K_{\lambda-n^{*}/2}(\cdot)$ is the modified Bessel function of the third kind.
\end{defi}

\subsection{Relationship to multivariate generalized asymmetric Laplace distribution}
\begin{theorem}
\label{vecaslap}
Let  $\boldsymbol{\mathcal{X}}$ is a random order-D tensor, with dimensional lengths $n_1\times n_2 \times \cdots \times n_D$, with realization $\boldsymbol{\mathfrak{X}}$. $\boldsymbol{\mathcal{X}}\sim\mathcal{TGAL}_{\mathbf{n}}\left(\boldsymbol{\mathfrak{M}},\mathbf{\Sigma}_1,\mathbf{\Sigma}_2,\ldots,\mathbf{\Sigma}_D,\lambda\right)\Leftrightarrow\operatorname{vec}\boldsymbol{\mathcal{X}}\sim\mathcal{GAL}_{n^{*}}\left(\operatorname{vec}\boldsymbol{\mathfrak{M}},\bigotimes_{i=1}^{D}\mathbf{\Sigma}_i,\lambda\right).$
\end{theorem}
\begin{proof}
\begin{equation*}
\begin{split}
f_{\boldsymbol{\mathcal{X}}}\left(\boldsymbol{\mathfrak{X}}\right)&=\frac{2\operatorname{exp}\left(\left(\operatorname{vec}\boldsymbol{\mathfrak{X}}\right)^{\mathrm{T}}\left(\bigotimes_{i=1}^{D}\mathbf{\Sigma}_i\right)^{-1}\operatorname{vec}\boldsymbol{\mathfrak{M}}\right)}{\left(2\pi\right)^{n^{*}/2}\Gamma\left(\lambda\right)\prod_{i=1}^{D}\left(\det\mathbf{\Sigma}_{i}\right)^{n^{*}/(2n_i)}}\left(\frac{\left(\operatorname{vec}\boldsymbol{\mathfrak{X}}\right)^{\mathrm{T}}\left(\bigotimes_{i=1}^{D}\mathbf{\Sigma}_i\right)^{-1}\operatorname{vec}\boldsymbol{\mathfrak{X}}}{2+\left(\operatorname{vec}\boldsymbol{\mathfrak{M}}\right)^{\mathrm{T}}\left(\bigotimes_{i=1}^{D}\mathbf{\Sigma}_i\right)^{-1}\operatorname{vec}\boldsymbol{\mathfrak{M}}}\right)^{\lambda/2-n^{*}/4}\\&\times K_{\lambda-n^{*}/2}\left(\sqrt{\left(2+\left(\operatorname{vec}\boldsymbol{\mathfrak{M}}\right)^{\mathrm{T}}\left(\bigotimes_{i=1}^{D}\mathbf{\Sigma}_i\right)^{-1}\operatorname{vec}\boldsymbol{\mathfrak{M}}\right)\left(\left(\operatorname{vec}\boldsymbol{\mathfrak{X}}\right)^{\mathrm{T}}\left(\bigotimes_{i=1}^{D}\mathbf{\Sigma}_i\right)^{-1}\operatorname{vec}\boldsymbol{\mathfrak{X}}\right)}\right),
\end{split}
\end{equation*}
where $n^{*}=\prod_{i=1}^{D}n_i$ and $K_{\lambda-n^{*}/2}(\cdot)$ is the modified Bessel function of the third kind.

By the definition of the multivariate generalized asymmetric Laplace distribution \cite{kozub}, $\operatorname{vec}\boldsymbol{\mathcal{X}}\sim\mathcal{GAL}_{n^{*}}\left(\operatorname{vec}\boldsymbol{\mathfrak{M}},\bigotimes_{i=1}^{D}\mathbf{\Sigma}_i,\lambda\right)$ if and only if
\begin{equation*}
\begin{split}
f_{\operatorname{vec}\boldsymbol{\mathcal{X}}}\left(\operatorname{vec}\boldsymbol{\mathfrak{X}}\right)&=\frac{2\operatorname{exp}\left(\left(\operatorname{vec}\boldsymbol{\mathfrak{X}}\right)^{\mathrm{T}}\left(\bigotimes_{i=1}^{D}\mathbf{\Sigma}_i\right)^{-1}\operatorname{vec}\boldsymbol{\mathfrak{M}}\right)}{\left(2\pi\right)^{n^{*}/2}\Gamma\left(\lambda\right)\prod_{i=1}^{D}\left(\det\mathbf{\Sigma}_{i}\right)^{n^{*}/(2n_i)}}\left(\frac{\left(\operatorname{vec}\boldsymbol{\mathfrak{X}}\right)^{\mathrm{T}}\left(\bigotimes_{i=1}^{D}\mathbf{\Sigma}_i\right)^{-1}\operatorname{vec}\boldsymbol{\mathfrak{X}}}{2+\left(\operatorname{vec}\boldsymbol{\mathfrak{M}}\right)^{\mathrm{T}}\left(\bigotimes_{i=1}^{D}\mathbf{\Sigma}_i\right)^{-1}\operatorname{vec}\boldsymbol{\mathfrak{M}}}\right)^{\lambda/2-n^{*}/4}\\&\times K_{\lambda-n^{*}/2}\left(\sqrt{\left(2+\left(\operatorname{vec}\boldsymbol{\mathfrak{M}}\right)^{\mathrm{T}}\left(\bigotimes_{i=1}^{D}\mathbf{\Sigma}_i\right)^{-1}\operatorname{vec}\boldsymbol{\mathfrak{M}}\right)\left(\left(\operatorname{vec}\boldsymbol{\mathfrak{X}}\right)^{\mathrm{T}}\left(\bigotimes_{i=1}^{D}\mathbf{\Sigma}_i\right)^{-1}\operatorname{vec}\boldsymbol{\mathfrak{X}}\right)}\right),
\end{split}
\end{equation*}
where $n^{*}=\prod_{i=1}^{D}n_i$ and $K_{\lambda-n^{*}/2}(\cdot)$ is the modified Bessel function of the third kind.

Thus, $\boldsymbol{\mathcal{X}}\sim\mathcal{TGAL}_{\mathbf{n}}\left(\boldsymbol{\mathfrak{M}},\mathbf{\Sigma}_1,\mathbf{\Sigma}_2,\ldots,\mathbf{\Sigma}_D,\lambda\right)\Leftrightarrow\operatorname{vec}\boldsymbol{\mathcal{X}}\sim\mathcal{GAL}_{n^{*}}\left(\operatorname{vec}\boldsymbol{\mathfrak{M}},\bigotimes_{i=1}^{D}\mathbf{\Sigma}_i,\lambda\right).$
\end{proof}
\begin{theorem}
\label{charmataslap}
Let $\boldsymbol{\mathfrak{T}}$ is an $\mathbf{n}$ dimensional order-D tensor. $\boldsymbol{\mathcal{X}}\sim\mathcal{TGAL}_{\mathbf{n}}\left(\boldsymbol{\mathfrak{M}},\mathbf{\Sigma}_1,\mathbf{\Sigma}_2,\ldots,\mathbf{\Sigma}_D,\lambda\right)$ if and only if the characteristic function of the random tensor is given by
\begin{equation*}
\varphi_{\boldsymbol{\mathcal{X}}}\left(\boldsymbol{\mathfrak{T}}\right)=\left(\frac{1}{1+\frac{1}{2}\left(\operatorname{vec}\boldsymbol{\mathfrak{T}}\right)^{\mathrm{T}}\left(\bigotimes_{i=1}^{D}\mathbf{\Sigma}_i\right)\operatorname{vec}\boldsymbol{\mathfrak{T}}-\mathrm{i}\left(\operatorname{vec}\boldsymbol{\mathfrak{M}}\right)^{\mathrm{T}}\operatorname{vec}\boldsymbol{\mathfrak{T}}}\right)^{\lambda}.
\end{equation*}
\end{theorem}
\begin{proof}
By Theorem \ref{vecaslap},
\begin{equation*}
\boldsymbol{\mathcal{X}}\sim\mathcal{TGAL}_{\mathbf{n}}\left(\boldsymbol{\mathfrak{M}},\mathbf{\Sigma}_1,\mathbf{\Sigma}_2,\ldots,\mathbf{\Sigma}_D,\lambda\right)\Leftrightarrow\operatorname{vec}\boldsymbol{\mathcal{X}}\sim\mathcal{GAL}_{n^{*}}\left(\operatorname{vec}\boldsymbol{\mathfrak{M}},\bigotimes_{i=1}^{D}\mathbf{\Sigma}_i,\lambda\right).
\end{equation*}
By the definition of the multivariate generalized asymmetric Laplace distribution \cite{kozub}, $\operatorname{vec}\boldsymbol{\mathcal{X}}\sim\mathcal{GAL}_{n^{*}}\left(\operatorname{vec}\boldsymbol{\mathfrak{M}},\bigotimes_{i=1}^{D}\mathbf{\Sigma}_i,\lambda\right)$ if and only if
\begin{equation*}
\varphi_{\operatorname{vec}\boldsymbol{\mathcal{X}}}\left(\operatorname{vec}\boldsymbol{\mathfrak{T}}\right)=\left(\frac{1}{1+\frac{1}{2}\left(\operatorname{vec}\boldsymbol{\mathfrak{T}}\right)^{\mathrm{T}}\left(\bigotimes_{i=1}^{D}\mathbf{\Sigma}_i\right)\operatorname{vec}\boldsymbol{\mathfrak{T}}-\mathrm{i}\left(\operatorname{vec}\boldsymbol{\mathfrak{M}}\right)^{\mathrm{T}}\operatorname{vec}\boldsymbol{\mathfrak{T}}}\right)^{\lambda}.
\end{equation*}
Thus,
\begin{equation*}
\varphi_{\operatorname{vec}\boldsymbol{\mathcal{X}}}\left(\operatorname{vec}\boldsymbol{\mathfrak{T}}\right)=\varphi_{\boldsymbol{\mathcal{X}}}\left(\boldsymbol{\mathfrak{T}}\right)=\left(\frac{1}{1+\frac{1}{2}\left(\operatorname{vec}\boldsymbol{\mathfrak{T}}\right)^{\mathrm{T}}\left(\bigotimes_{i=1}^{D}\mathbf{\Sigma}_i\right)\operatorname{vec}\boldsymbol{\mathfrak{T}}-\mathrm{i}\left(\operatorname{vec}\boldsymbol{\mathfrak{M}}\right)^{\mathrm{T}}\operatorname{vec}\boldsymbol{\mathfrak{T}}}\right)^{\lambda}.
\end{equation*}
\end{proof}
\begin{theorem}
Let $\boldsymbol{\mathcal{Y}}\sim\mathcal{TGAL}_{\mathbf{n}}\left(\boldsymbol{\mathfrak{M}},\mathbf{\Sigma}_1,\mathbf{\Sigma}_2,\ldots,\mathbf{\Sigma}_D,\lambda\right)$, $\boldsymbol{\mathcal{X}}\sim\mathcal{TN}_{\mathbf{n}}\left(\mathbb{O},\mathbf{\Sigma}_1,\mathbf{\Sigma}_2,\ldots,\mathbf{\Sigma}_D\right)$ and $W$ has the standart gamma distribution with shape parameter $\lambda$, independent of $\boldsymbol{\mathcal{X}}$, then 
\begin{equation*}
\boldsymbol{\mathcal{Y}}=\boldsymbol{\mathfrak{M}}W+W^{1/2}\boldsymbol{\mathcal{X}}.
\end{equation*}
\end{theorem}
\begin{proof}
Let $\boldsymbol{\mathcal{Y}}=\boldsymbol{\mathfrak{M}}W+W^{1/2}\boldsymbol{\mathcal{X}}$.

Therefore,
\begin{equation*}
\operatorname{vec}\boldsymbol{\mathcal{Y}}=\operatorname{vec}\left(\boldsymbol{\mathfrak{M}}W+W^{1/2}\boldsymbol{\mathcal{X}}\right)=\left(\operatorname{vec}\boldsymbol{\mathfrak{M}}\right)W+W^{1/2}\operatorname{vec}\boldsymbol{\mathcal{X}}.
\end{equation*}
By the definition of the tensor normal distribution \cite[Theorem 2.1]{galtensor},
\begin{equation*}
\boldsymbol{\mathcal{X}}\sim\mathcal{TN}_{\mathbf{n}}\left(\boldsymbol{\mathfrak{M}},\mathbf{\Sigma}_1,\mathbf{\Sigma}_2,\ldots,\mathbf{\Sigma}_D\right)\Leftrightarrow\operatorname{vec}\boldsymbol{\mathcal{X}}\sim\mathcal{N}_{n^{*}}\left(\operatorname{vec}\boldsymbol{\mathfrak{M}},\bigotimes_{i=1}^{D}\mathbf{\Sigma}_i\right).
\end{equation*}
By the theorem \cite{kozub},
\begin{equation*}
\operatorname{vec}\boldsymbol{\mathcal{Y}}\sim\mathcal{GAL}_{n^{*}}\left(\operatorname{vec}\boldsymbol{\mathfrak{M}},\bigotimes_{i=1}^{D}\mathbf{\Sigma}_i,\lambda\right).
\end{equation*}
By Theorem \ref{vecaslap},
\begin{equation*}
\operatorname{vec}\boldsymbol{\mathcal{Y}}\sim\mathcal{GAL}_{n^{*}}\left(\operatorname{vec}\boldsymbol{\mathfrak{M}},\bigotimes_{i=1}^{D}\mathbf{\Sigma}_i,\lambda\right)\Leftrightarrow \boldsymbol{\mathcal{Y}}\sim\mathcal{TGAL}_{\mathbf{n}}\left(\boldsymbol{\mathfrak{M}},\mathbf{\Sigma}_1,\mathbf{\Sigma}_2,\ldots,\mathbf{\Sigma}_D,\lambda\right).
\end{equation*}
\end{proof}
\begin{theorem}
\label{expvalueaslap}
If $\boldsymbol{\mathcal{X}}\sim\mathcal{TGAL}\left(\boldsymbol{\mathfrak{M}},\mathbf{\Sigma}_1,\mathbf{\Sigma}_2,\ldots,\mathbf{\Sigma}_D,\lambda\right)$, then the expected value of the random tensor $\boldsymbol{\mathcal{X}}$ is given by
\begin{equation*}
 \mathbb{E}\left[\boldsymbol{\mathcal{X}}\right]=\lambda\boldsymbol{\mathfrak{M}}.
\end{equation*}
\end{theorem}
\begin{proof}
By Theorem \ref{vecaslap},
\begin{equation*}
\boldsymbol{\mathcal{X}}\sim\mathcal{TGAL}_{\mathbf{n}}\left(\boldsymbol{\mathfrak{M}},\mathbf{\Sigma}_1,\mathbf{\Sigma}_2,\ldots,\mathbf{\Sigma}_D,\lambda\right)\Leftrightarrow\operatorname{vec}\boldsymbol{\mathcal{X}}\sim\mathcal{GAL}_{n^{*}}\left(\operatorname{vec}\boldsymbol{\mathfrak{M}},\bigotimes_{i=1}^{D}\mathbf{\Sigma}_i,\lambda\right).
\end{equation*}
By the property of the multivariate generalized asymmetric Laplace distribution \cite{kozub}, if $\operatorname{vec}\boldsymbol{\mathcal{X}}\sim\mathcal{GAL}_{n^{*}}\left(\operatorname{vec}\boldsymbol{\mathfrak{M}},\bigotimes_{i=1}^{D}\mathbf{\Sigma}_i,\lambda\right)$, then
\begin{equation*}
 \mathbb{E}\left[\operatorname{vec}\boldsymbol{\mathcal{X}}\right]=\lambda\operatorname{vec}\boldsymbol{\mathfrak{M}}=\operatorname{vec}\left(\lambda\boldsymbol{\mathfrak{M}}\right).
\end{equation*}
Thus,
\begin{equation*}
 \mathbb{E}\left[\boldsymbol{\mathcal{X}}\right]=\lambda\boldsymbol{\mathfrak{M}}.
\end{equation*}
\end{proof}
\begin{theorem}
\label{kxxaslap}
If  $\boldsymbol{\mathcal{X}}\sim\mathcal{TGAL}_{\mathbf{n}}\left(\boldsymbol{\mathfrak{M}},\mathbf{\Sigma}_1,\mathbf{\Sigma}_2,\ldots,\mathbf{\Sigma}_D,\lambda\right)$, then the variance-covariance matrix of the random tensor $\boldsymbol{\mathcal{X}}$ is given by
\begin{equation*}
\mathbf{K}_{\boldsymbol{\mathcal{XX}}}=\lambda\left(\bigotimes_{i=1}^{D}\mathbf{\Sigma}_i+\left(\operatorname{vec}\boldsymbol{\mathfrak{M}}\right)\left(\operatorname{vec}\boldsymbol{\mathfrak{M}}\right)^{\mathrm{T}}\right).
\end{equation*}
\end{theorem}
\begin{proof}
By Theorem \ref{vecaslap},
\begin{equation*}
\boldsymbol{\mathcal{X}}\sim\mathcal{TGAL}_{\mathbf{n}}\left(\boldsymbol{\mathfrak{M}},\mathbf{\Sigma}_1,\mathbf{\Sigma}_2,\ldots,\mathbf{\Sigma}_D,\lambda\right)\Leftrightarrow\operatorname{vec}\boldsymbol{\mathcal{X}}\sim\mathcal{GAL}_{n^{*}}\left(\operatorname{vec}\boldsymbol{\mathfrak{M}},\bigotimes_{i=1}^{D}\mathbf{\Sigma}_i,\lambda\right).
\end{equation*}
By the property of the multivariate generalized asymmetric Laplace distribution \cite{kozub}, if $\operatorname{vec}\boldsymbol{\mathcal{X}}\sim\mathcal{GAL}_{n^{*}}\left(\operatorname{vec}\boldsymbol{\mathfrak{M}},\bigotimes_{i=1}^{D}\mathbf{\Sigma}_i,\lambda\right)$, then
\begin{equation*}
\mathbf{K}_{\operatorname{vec}\boldsymbol{\mathcal{X}},\operatorname{vec}\boldsymbol{\mathcal{X}}}=\lambda\left(\bigotimes_{i=1}^{D}\mathbf{\Sigma}_i+\left(\operatorname{vec}\boldsymbol{\mathfrak{M}}\right)\left(\operatorname{vec}\boldsymbol{\mathfrak{M}}\right)^{\mathrm{T}}\right).
\end{equation*}
Thus,
\begin{equation*}
\mathbf{K}_{\boldsymbol{\mathcal{XX}}}=\mathbf{K}_{\operatorname{vec}\boldsymbol{\mathcal{X}},\operatorname{vec}\boldsymbol{\mathcal{X}}}=\lambda\left(\bigotimes_{i=1}^{D}\mathbf{\Sigma}_i+\left(\operatorname{vec}\boldsymbol{\mathfrak{M}}\right)\left(\operatorname{vec}\boldsymbol{\mathfrak{M}}\right)^{\mathrm{T}}\right).
\end{equation*}
\end{proof}
\section{Conclusions}
In this article, we have defined the matrix variate asymmetric Laplace distribution. We have proved some properties of the matrix variate asymmetric Laplace distribution. We have proved the relationship between the matrix variate asymmetric Laplace distribution and the multivariate asymmetric Laplace distribution.

Based on all of the above, we can argue that the matrix variate asymmetric Laplace distribution is a generalization of the multivariate asymmetric Laplace distribution.

In this article, we have defined the matrix variate generalized asymmetric Laplace distribution. We have proved some properties of the matrix variate generalized asymmetric Laplace distribution. We have proved the relationship between the matrix variate generalized asymmetric Laplace distribution and the multivariate generalized asymmetric Laplace distribution. We have proved the relationship between the matrix variate generalized asymmetric Laplace distribution and the matrix variate asymmetric Laplace distribution.

Based on all of the above, we can argue that the matrix variate generalized asymmetric Laplace distribution is a generalization of the multivariate generalized asymmetric Laplace distribution. Also we can argue that the matrix variate generalized asymmetric Laplace distribution is a generalization of the matrix variate asymmetric Laplace distribution.

We have obtained results that do not contradict each other and agree with the theory of matrix distributions.

In this article, we have defined the tensor variate asymmetric Laplace distribution. We have proved some properties of the tensor variate asymmetric Laplace distribution. We have proved the relationship between the tensor variate asymmetric Laplace distribution and the multivariate asymmetric Laplace distribution. Also we can argue that the matrix variate asymmetric Laplace distribution is a special case of the tensor variate asymmetric Laplace distribution if $D=2$.

Based on all of the above, we can argue that the tensor variate asymmetric Laplace distribution is a generalization of the multivariate asymmetric Laplace distribution.

In this article, we have defined the tensor variate generalized asymmetric Laplace distribution. We have proved some properties of the tensor variate generalized asymmetric Laplace distribution. We have proved the relationship between the tensor variate generalized asymmetric Laplace distribution and the multivariate generalized asymmetric Laplace distribution. Also we can argue that the matrix variate generalized asymmetric Laplace distribution is a special case of the tensor variate generalized asymmetric Laplace distribution if $D=2$.

Based on all of the above, we can argue that the tensor variate generalized asymmetric Laplace distribution is a generalization of the multivariate generalized asymmetric Laplace distribution. Also we can argue that the tensor variate generalized asymmetric Laplace distribution is a generalization of the tensor variate asymmetric Laplace distribution.

We have obtained results that do not contradict each other and agree with the theory of tensor distributions.

\end{document}